\documentclass[11pt]{amsart}

\usepackage{amsmath,amssymb,latexsym,enumerate}
\usepackage{a4wide}
\usepackage{color,graphicx}
\usepackage{ifthen}
\usepackage{verbatim}
\usepackage{fancyhdr}
\usepackage{url}
\usepackage{bbm}
\usepackage{graphicx}
\usepackage[sans]{dsfont}
\usepackage{mathrsfs}

\renewcommand{\Im}{\operatorname{Im}}

\newcommand{\diag}{\operatorname{diag}}
\newcommand{\dist}{\operatorname{dist}}
\newcommand{\esssupp}{\operatorname{ess-supp}}
\newcommand{\hess}{\operatorname{Hess}}
\newcommand{\Id}{\operatorname{Id}}
\newcommand{\Op}{\operatorname{Op}}
\newcommand{\rank}{\operatorname{Rank}}
\newcommand{\supp}{\operatorname{supp}}
\newcommand{\vect}{\operatorname{Vect}}
\newcommand{\vol}{\operatorname{vol}}

\def\<{\langle}
\def\>{\rangle}

\def\one{\mathds{1}}

\newcommand{\Subsection}[1]{\subsection{ #1} ${}^{}$}

\newcommand{\A}{\mathcal A}
\newcommand{\B}{\mathcal B}
\newcommand{\D}{\mathcal D}

\newcommand{\F}{\mathcal F}
\newcommand{\J}{\mathcal J}
\newcommand{\U}{\mathcal U}

\newcommand{\C}{\mathbb C}
\newcommand{\N}{\mathbb N}
\newcommand{\R}{\mathbb R}

\newcommand{\M}{\mathscr M}

\newcommand{\be}{\begin{equation}}
\newcommand{\ee}{\end{equation}}
\newcommand{\bes}{\begin{equation*}}
\newcommand{\ees}{\end{equation*}}

\newcommand{\Dh}{d_{\phi,h}}

\newcommand{\Dhs}{d_{\phi,h}^*}
\renewcommand{\d}{\operatorname{d}}

\numberwithin{equation}{section}
\numberwithin{figure}{section}

\newtheorem{theorem}{Theorem}[section]
\newtheorem{corollary}[theorem]{Corollary}
\newtheorem{lemma}[theorem]{Lemma}
\newtheorem{definition}[theorem]{Definition}
\newtheorem{proposition}[theorem]{Proposition}
\newtheorem{remark}[theorem]{Remark}

\newtheorem{hyp}{Hypothesis}

\def\aaa{{\mathcal A}}
\def\eee{{\mathcal E}}\def\fff{{\mathcal F}} \def\hhh{{\mathcal H}}
 \def\lll{{\mathcal L}}
\def\mmm{{\mathcal M}} \def\ooo{{\mathcal O}}
\def\sss{{\mathcal S}}
\def\uuu{{\mathcal U}}

\def\C{\mathbb C}   \def\G{\mathbb G} 
 
 \def\N{\mathbb N}    
 \def\R{\mathbb R}   

\def\D{\Xi}
\def\d{\textbf{d}}
\newcommand{\T}{{\bf T}}

\def\m{\mathbf{m}}
\def\s{\mathbf{s}}
\def\u{\mathbf{u}}

\begin{document}
\title{Tunnel effect for semiclassical random walk}

\author[J.-F. Bony]{Jean-Fran\c{c}ois~Bony}
\address{Institut  Math\'ematiques de Bordeaux \\ Universit\'e de Bordeaux}
\email{bony@\allowbreak math.u-bordeaux1.\allowbreak fr}

\author[F. H\'erau]{Fr\'ed\'eric~H\'erau}
\address{Laboratoire de Math\'ematiques Jean Leray \\ Universit\'e de Nantes}
\email{ frederic.herau@\allowbreak univ-nantes.\allowbreak fr}

\author[L. Michel]{Laurent~Michel}
\address{Laboratoire J.-A. Dieudonn\'e\\ Universit\'e de Nice}
\email{lmichel@\allowbreak unice.\allowbreak fr}

\begin{abstract}
We study  a semiclassical random walk with respect to a probability measure with a finite number $n_0$ of wells. We show that the associated operator has  exactly $n_0$ exponentially close to $1$ eigenvalues (in the semiclassical sense), and that the other are $\ooo(h)$ away from $1$. We also give an asymptotic of these small eigenvalues.
The key ingredient in our approach is a general factorization result of pseudodifferential operators, which allows us to use recent results on the Witten Laplacian.
\thanks{J.-F. Bony and F. H\'erau are supported by the ANR project NOSEVOL, ANR 2011 BS01019 01. L. Michel is member of the ERC project: Semi Classical Analysis of Partial Differential
Equations,
ERC-2012-ADG, project number 320845}

\end{abstract}
\maketitle

\section{Introduction} \label{z6}
Let  $\phi:\R^d\rightarrow\R$ be a smooth function and let $h\in]0,1]$ denote a small parameter in all the paper.
Under suitable assumptions specified later, the density $e^{-\phi(x)/h}$ is integrable and there exists $Z_h>0$ such that $d\mu_h(x)=Z_he^{-\phi(x)/h}dx$ defines a probability measure on $\R^d$. We can associate to $\mu_h$ the Markov kernel $t_h(x,dy)$ given by
\begin{equation}\label{e1}
t_h(x,dy)=\frac 1{\mu_h(B(x,h))}\one_{\vert x-y\vert<h}d\mu_h(y).
\end{equation}

From the point of view of random walks, this kernel can be understood  as follows: assume at step $n$, the walk is in $x_n$, then the point $x_{n+1}$ is
choosen in the small ball $B(x_n,h)$, uniformly at random with respect to $d\mu_h$. The probability distribution at time $n\in\N$ of a walk starting from $x$ is given by
the kernel $t_h^n(x,dy)$. The long time behavior ($n\rightarrow \infty$) of the kernel $t_h^n(x,dy)$ carries informations on the ergodicity of the random walk,
and has many practical applications (we refer to \cite{LeRoSt10_01} for an overview of computational aspects).
Observe that if $\phi$ is a Morse function, then  the density $e^{-\phi/h}$ concentrates at scale $\sqrt{h}$ around minima of $\phi$, whereas the moves of
the random walk are at scale $h$.

Another point of view comes from statistical physics and can be described as follows. One can associate to the kernel $t_h(x,dy)$ an operator $\T_h$ acting on the space $C_0$ of continuous functions
going to zero at infinity, by the formula
\begin{equation*}
\T_hf(x)=\int_{\R^d}f(y)t_h(x,dy)=\frac 1{\mu_h(B(x,h))}\int_{\vert x-y\vert<h}f(y)d\mu_h(y).
\end{equation*}
This defines a bounded operator on $C_0$, enjoying the Markov property $(\T_h(1)=1)$.

The transpose $\T_h^\star$ of $\T_h$  is defined by duality on the set of bounded positive measures ${\bf M}^+_b$ (resp. bounded measures ${\bf M}_b$).
If $d\nu$ is a bounded measure we have
\begin{equation}\label{e2}
\T_h^\star(d \nu)= \bigg(  \int_{\R^d} \one_{\vert x-y\vert<h} \mu_h(B(y,h))^{-1} d\nu(y) \bigg) d\mu_h.
\end{equation}
Assume that a particle in $\R^d$ is distributed according to a probability measure $d\nu$, then $\T_h^\star(d\nu)$ represents its distribution after a move according to $t_h(x,dy)$, and
the distribution after $n$ steps is then given by  $(\T_h^\star)^n(d\nu)$. The existence of a limit distribution is strongly related to the existence of an invariant measure.
In the present context, one can easily see that $\T_h^\star$ admits the following invariant measure
\begin{equation*}
d\nu_{h,\infty}(x)=\widetilde Z_h\mu_h(B(x,h))d\mu_h(x),
\end{equation*}
where $\widetilde Z_h$ is chosen so that $d\nu_{h,\infty}$ is a probability. The aim of the present paper will be to prove the convergence of
$(\T_h^\star)^n(d\nu)$ towards $d\nu_{h,\infty}$ when $n$ goes to infinity, for any probability measure $d\nu$, and to get precise informations on the  speed of convergence.
Taking $d\nu(y)=\delta_x(y)$, it turns out that this is equivalent to study the convergence of $t_h^n(x,dy)$ towards $d\nu_{h,\infty}$.
Observe that in the present setting, proving pointwise convergence ($h$ being fixed) of $t_h^n(x,dy)$ towards the invariant measure is an easy consequence of some general theorem
(see \cite{Fe71_01}, Theorem 2, p272). The interest of our approach is to get convergence in a stronger topology and to obtain precise information on the behavior with respect to the semiclassical parameter $h$.

Before going further, let us recall some elementary properties of $\T_h$ that will be usefull in the sequel. First, we can see easily from its definition that the operator
$\T_h$ can be extended as a bounded operator both on $L^\infty(d\nu_{h,\infty})$ and $L^1(d\nu_{h,\infty})$. From the Markov property and the fact that $d\nu_{h,\infty}$ is stationary
it is clear that
\begin{equation*}
\|\T_h\|_{L^\infty(d\nu_{h,\infty})\rightarrow L^\infty(d\nu_{h,\infty})}= \|\T_h\|_{L^1(d\nu_{h,\infty})\rightarrow L^1(d\nu_{h,\infty})}=1.
\end{equation*}
Hence, by interpolation $\T_h$ defines  also a bounded operator of norm $1$ on $L^2(\R^d,d\nu_{h,\infty})$. Finally, observe that
$\T_h$ is selfadjoint on $L^2(d\nu_{h,\infty})$ (thanks again to Markov property).

\medskip

Let us go back to the study of the sequence $(\T_h^\star)^n$ and explain the topology we use to study the convergence of this sequence. Instead of looking at this evolution on the full set of bounded measures,
we restrict the analysis by introducing the following stable Hilbert space
\begin{equation} \label{a10}
\hhh_h = L^2 ( d \nu_{h , \infty} ) = \Big\{ f \text{ measurable on } \R^{d} \text{ such that } \int |f(x)|^2 \, d \nu_{h , \infty}  < \infty \Big\}.
\end{equation}
for which we have a natural injection with norm $1$,
$
 \J:\hhh_h  \hookrightarrow {\bf M}_b,
$
when identifying an absolutely continuous measure $d\nu_h = f(x) d\nu_{h,\infty}$ with its density $f$.
Using \eqref{e2}, we can see easily that $\T_h^\star\circ\J=\J\circ \T_h$. From this identification $\T_h^\star$ (acting on $\hhh_h$) inherits the properties
of $\T_h$:
\begin{equation} \label{a9}
\T^{\star}_{h} : \hhh_h \longrightarrow \hhh_h \text{ is selfadjoint and continuous with operator norm } 1 .
\end{equation}
 Hence, its spectrum is contained in the interval $[-1,1]$. Moreover, we will see later that $(-1)$ is sufficiently far from the spectrum. Since we are interested in the convergence of $(\T_h^\star)^n$ in $L^2$ topology, it is then sufficient
for our purpose to  give a precise description of the spectrum of $\T_h$ near $1$.

\medskip

Convergence of Markov chains to stationary distribution is a wide area of research and applications.
Knowing that a computable Markov kernel converges to a given distribution may be very useful in practice. In particular it is often used to sample
a given probability in order to implement Monte-Carlo methods (see  \cite{LeRoSt10_01} for numerous algorithms and computational aspects). However, most of results giving a priori bound on the speed of convergence for such algorithms holds for discrete state space
 (we refer to  \cite{Di09_01} for a state of the art on Monte-Carlo-Markov-Chain methods).

 This point of view is also used to track extremal points of any function by simulated annealing procedure. For example, this was used in \cite{HoSt88_01} on finite state space and in
 \cite{HoKuSt89_01}, \cite{Mi92_01} on continuous state space.

Eventually, let us recall that the study of time continuous processes is of current interest in statistical physics (see for instance the work of Bovier-Gayrard-Klein on metastable states \cite{BoEcGaKl04_01,BoGaKl05_01}).

More recently, Diaconis-Lebeau obtained first results on discrete time processes on continuous state space \cite{DiLe09_01}.
This approach was then further developed in \cite{DiLeMi11_01} to get convergence results on the Metropolis algorithm on bounded domains of the Euclidean space.
Similar results were also obtained in \cite{LeMi10_01},  \cite{GuMi11_01} in various geometric situations.
In all these papers, the probability $d\mu_h$ is independant of $h$ which leads \emph{in fine} to a spectral gap of order $h^2$.
Here, the situtation is quite different and somehow ``more semiclassical''. This permits to exhibit situations with very small spectral gap of order
$e^{-c/h}$. In this paper, we shall compute accurately this spectral gap under the following asumptions on $\phi$.

\begin{hyp}\sl \label{h1} We suppose that $\phi$ is a Morse function, with non-degenerate critical points, and that there exists $c,R>0$ and some constants $C_\alpha>0$, $\alpha\in\N^d$ such that for all $|x|\geq R$, we have
\begin{equation*}
\forall \alpha \in \N^d \setminus \{ 0 \} ,  \qquad \vert \partial_x^\alpha \phi ( x ) \vert \leq C_\alpha , \quad \vert \nabla \phi (x) \vert \geq c \quad \text{and} \quad \vert \phi (x) \vert \geq c \vert x \vert .
\end{equation*}
  In particular, there is a finite number of critical points.
 \end{hyp}

 Observe that functions $\phi$ satisfying this assumption are at most linear at infinity. It may be possible  to relax this assumption to quadratic growth at infinity and we guess our results hold true also in this context. However, it  doesn't seem possible to get a complete proof with the class of symbols used in this paper.

 Under the above assumption, it is clear that $d\mu_h(x)=Z_he^{-\phi(x)/h}dx$ is a probability measure.
 For the following we call $\uuu$ the set of critical points $\u$. We denote by $\uuu^{(0)}  $
 the set of minima of $\phi$  and $\uuu^{(1)} $ the set of  saddle points, i.e. the critical points with index $1$  (note that this set may be empty). We also introduce
 $n_j=\sharp\, \uuu^{(j)}  $, $j=0,1$, the number of elements of $\uuu^{(j)}$.

We shall prove first the following result:

\begin{theorem}\sl \label{a1}
There exist $\delta , h_{0} > 0$ such that the following assertions hold true for
$h \in ] 0 , h_{0} ]$. First, $\sigma ( \T_h^\star ) \subset [ - 1 + \delta , 1 ]$ and $\sigma_{ess} ( \T_h^\star ) \subset [ - 1 + \delta , 1 - \delta ]$.
Moreover, $\T_h^\star$ has exactly $n_{0}$ eigenvalues in $[ 1 - \delta h , 1]$ which are in fact in $[ 1 - e^{- \delta / h} , 1]$.
Eventually, $1$ is a simple eigenvalue for the eigenstate $\nu_{h,\infty}\in\hhh_h$.
\end{theorem}

 This theorem will be proved in the next section.
 The goal of this paper is to describe accurately the eigenvalues  close to $1$. We will see later, that describing the eigenvalues of $\T_h^\star$ close to $1$
 has many common points with the spectral study of the so-called 
 semiclassical Witten Lapacian (see section 4). We introduce the following generic assumptions on the critical points of $\phi$.

\begin{hyp}\sl \label{h2}
We suppose that the values $\phi ( \s ) - \phi ( \m )$ are distinct for any $\s\in \uuu^{(1)}$ and $\m\in \uuu^{(0)}$.
 \end{hyp}

 Note now that this generic assumption could easily be relaxed, under rather loud additional notations and less precise statements, following e.g. \cite{HeHiSj11_01}, and that we chose to focus in this article on other particularities of the problem we study.

 Let us recall that under the above assumptions, there exists a labeling of minima and saddle points: $\uuu^{(0)} = \{ \m_k ; \ k = 1 , \ldots , n_{0} \} $
  and $\uuu^{(1)} = \{ \s_j ; \ j=2, \ldots , n_1+1 \}$ which permits  to describe the low liying eigenvalues of the Witten Laplacian (see \cite{HeKlNi04_01}, \cite{HeHiSj11_01} for instance). Observe that the
  enumeration of $\uuu^{(1)}$ starts with $j=2$ since we will need a fictive saddle point $\s_1=+\infty$.
 We shall recall this labeling procedure in the appendix.

 Let us denote $1=\lambda_1^\star(h)> \lambda_2^\star(h)\geq\cdots \geq \lambda^\star_{n_0}(h)$ the $n_0$ largest eigenvalues of $\T_h^\star$.
  The main result of this paper is the following

\begin{theorem}\sl \label{e3}
Under Hypotheses \ref{h1} and \ref{h2}, there exists a labeling of minima and saddle points and constants $\alpha , h_0>0$ such that, for all $k = 2, \ldots , n_0$ and for any $h \in ] 0 , h_0 ]$,
\begin{equation*}
 1- \lambda_k^\star(h) = \frac{h}{(2d+4)\pi}\mu_k\sqrt{\bigg\vert \frac{\det \phi''(\m_k)}{\det \phi''(\s_k)} \bigg\vert} e^{- 2 S_k/ h}(1+\ooo(h)) ,
\end{equation*}
where $S_k : = \phi(\s_k) - \phi(\m_k)$ (Aarhenius number) and $-\mu_k$ denotes the unique negative eigenvalue of $\phi''$ at $\s_k$.
\end{theorem}

\begin{remark}\sl
Observe that the leading term in the asymptotic of $1- \lambda_k^\star(h)$ above is exactly (up to the factor $(2d+4)$) the one of the $k$-th eigenvalue of the Witten Laplacian on the $0$-forms
obtained in \cite{HeKlNi04_01}. This relationship will be transparent from the proof below.
\end{remark}

As an immediate consequence of these results and of the spectral theorem, we get that the convergence to equilibrium holds slowly and that the system has a metastable regime. More precisely, we have the following result whose proof can be found at the end of Section \ref{z4}.

\begin{corollary}\sl \label{z5}
Let $d \nu_{h}$ be probability measure in $\hhh_{h}$ and assume first that $\phi$ has a unique minimum. Then, using that $\sigma ( \T_{h}^{\star} ) \subset [ - 1 + \delta , 1 - \delta h ]$, it yields
\begin{equation} \label{z3}
\big\Vert ( \T_{h}^{\star} )^{n} ( d \nu_{h} ) - d \nu_{h , \infty} \big\Vert_{\hhh_h} = \ooo (h) \Vert d \nu_{h} \Vert_{\hhh_h} .
\end{equation}
for all $n \gtrsim \vert \ln h \vert h^{-1}$ which corresponds to the Ehrenfest time. But, if $\phi$ has now several minima, we can write
\begin{equation} \label{z2}
( \T_{h}^{\star} )^{n} ( d \nu_{h} ) =  \Pi d \nu_{h} + \ooo (h) \Vert d \nu_{h} \Vert_{\hhh_h} ,
\end{equation}
for all $h^{-1} \vert \ln h \vert \lesssim n \lesssim e^{2S_{n_0}/ h}$. Here, $\Pi$ can be taken as the orthogonal projector on the $n_{0}$ functions $\chi_{k} ( x ) e^{- ( \phi (x) - \phi ( \m_{k} ) ) / h}$ where $\chi_{k}$ is any cutoff function near $\m_{k}$.

On the other hand, we have, for any $n \in \N$,
\begin{equation} \label{z1}
\big\Vert ( \T_{h}^{\star} )^{n} ( d \nu_{h} ) - d \nu_{h , \infty} \big\Vert_{\hhh_h} \leq ( \lambda_{2}^{\star} (h) )^{n} \Vert d \nu_{h} \Vert_{\hhh_h} ,
\end{equation}
where $\lambda_{2}^{\star} (h)$ is described in Theorem \ref{e3}. Note that this inequality is optimal. In particular, for $n \gtrsim \vert \ln h \vert h^{-1} e^{2 S_{2} / h}$, the right hand side of \eqref{z1} is of order $\ooo (h) \Vert d \nu_{h} \Vert_{\hhh_h}$.

Thus, for a reasonable number of iterations (which guaranties \eqref{z3}), $1$ seems to be an eigenvalue of multiplicity $n_{0}$; whereas, for a very large number of iterations, the system returns to equilibrium. Then, \eqref{z2} is a metastable regime.
\end{corollary}

Since $t_h(x,dy)$ is absolutely continuous with respect to $d\nu_{h,\infty}$, then $(\T_h^\star)^n(\delta_{y=x})=t_h^{n}(x,dy)$ belongs to $\mathcal{H}_h$ for any $n\geq 1$. Hence, the above estimate and the fact that $d\nu_{h,\infty}$ is invariant show that
\begin{equation*}
\big\Vert t_h^n(x,dy)-d\nu_{h,\infty} \big\Vert_{\mathcal{H}_h}\leq ( \lambda_2^\star (h) )^{n-1}\Vert t_h(x,dy)\Vert_{\mathcal{H}_h}.
\end{equation*}
Moreover the prefactor $\Vert t_h(x,dy)\Vert_{\mathcal{H}_h}$ could be easily computed but depends on $x$ and $h$.

\medskip

Throughout this paper, we use semiclassical analysis (see \cite{DiSj99_01}, \cite{Ma02_01}, or \cite{Zw12_01} for expository books of  this theory).
  Let us recall that a function $m:\R^d\rightarrow \R^+$ is an order function if there exists $N_0\in\N$ and a constant $C>0$ such that
  for all $x,y\in\R^d,\,m(x)\leq C\<x-y\>^{N_0}m(y)$. This definition can be extended to functions  $m:\R^d\times\C^{d'}\rightarrow\R^+$ by identifying $\R^d\times\C^{d'}$ with
  $\R^{d+2d'}$.
Given on order function $m$ on $T^*\R^d\simeq\R^{2d}$, we will denote by $S^0(m)$ the space of semiclassical symbols on $T^*\R^d$ whose all derivatives are bounded by $m$ and
$\Psi^0(m)$ the set of corresponding pseudodifferential operators. For any $\tau\in]0,\infty]$, and any order function $m$ on $\R^d\times\C^d$
we will denote by $S^0_\tau(m)$ the set of symbols which are
analytic with respect  to $\xi$ in the strip $\vert \Im \xi\vert<\tau$ and bounded by some constant times $m(x,\xi)$ in this strip.
We will denote by $S^0_\infty(m)$ the union for $\tau>0$ of $S^0_\tau(m)$.
We denote by $\Psi^0_\tau(m)$ the set of corresponding operators.
Eventually, we say that  a symbol $p$ is classical if it admits  an asymptotic expansion $p(x,\xi;h)\sim\sum_{j\geq 0}h^jp_j(x,\xi)$. We will denote by $S^0_{\tau,cl}(m)$, $S^0_{cl}(m)$ the corresponding class of symbols.

 We will also need some matrix valued pseudodifferential operators.
 Let $\M_{p,q}$ denote the set of real valued matrices with $p$ rows and $q$ columns and $\M_p=\M_{p,p}$.
 Let $\A:T^*\R^d\rightarrow \M_{p,q}$ be a smooth function. We will say that $\A$ is a $(p,q)$-matrix-weight
 if $\A(x,\xi)=(a_{i,j}(x,\xi))_{i,j}$ and for any $i=1\ldots,p$ and $j=1,\ldots,q$, $a_{i,j}$ is an order function.
 If $p=q$, we will simply say that $\A$ is $q$-matrix-weight.

 Given a $(p,q)$-matrix-weight  $\A$, we  will denote by $S^0(\A)$ the set of symbols $p(x,\xi)=(p_{i,j}(x,\xi))_{i,j}$ defined on $T^*\R^d$ with
values in $\M_{p,q}$ such that for all $i,j$, $p_{i,j}\in S^0(a_{i,j})$ and $\Psi^0(\M_{p,q})$ the set of corresponding pseudodifferential operators.
Obvious extensions of this definitions leads to the definition of matrix valued symbol analytic w.r.t. to $\xi$ and the corresponding operators: $S^0_\tau(\A)$ and $\Psi_\tau^0(\A)$.
In the following, we shall mainly use the Weyl semiclassical quantization of symbols,
defined by
\be
\Op(p)u(x)=(2\pi h)^{-d}\int_{T^*\R^d} e^{ih^{-1}(x-y)\xi}p(\frac {x+y}2,\xi)u(y)dyd\xi
\ee
for $p \in S^0(\A)$. We shall also use the following notations all along the paper. Given two pseudo differential operators $A$ and $B$, we shall write
$A=B+\Psi^k(m)$ if the difference $A-B$ belongs to $\Psi^k(m)$. At the level of symbols, we shall write $a=b+S^k(m)$ instead of $a-b\in S^k(m)$.

\medskip

The preceding theorem is  close - in the spirit and in the proof -   to the ones given for the Witten Laplacian in \cite{HeKlNi04_01} and for
the Kramers Fokker Planck operators in \cite{HeHiSj11_01}. In those works, the results are deeply linked with some properties inherited from a so-called supersymmetric structure, allowing to write the operators as twisted Hodge Laplacians of the form
\begin{equation*}
P = d_{\phi, h}^* A d_{\phi, h}
\end{equation*}
where $\d$ is  the usual differential, $d_{\phi, h} = h\d + d \phi(x) \wedge = e^{-\phi/h} h \d e^{\phi/h}$ is the  differential twisted by $\phi$, and $A$ is a constant matrix in $\M_d$.
Here we are able to recover a supersymmetric type structure, and the main ingredients for the study of the exponentially small eigenvalues are therefore available. This is contained in the following theorem, that we give in rather general context since it may be useful in other situations.

Let us introduce  the $d$-matrix-weight, $\D,\aaa:T^*\R^d\rightarrow \M_d$ given by $\aaa_{i,j}(x,\xi)=(\<\xi_i\>\<\xi_j\>)^{-1}$,
$\Xi_{i,j}=\delta_{i,j}\<\xi_i\>$ and observe that  $(\Xi\A)_{i,j}=\<\xi_j\>^{-1}$. In the following theorem, we state an exact factorization result which will be the key point in our approach.

\begin{theorem}\sl \label{e4}
Let $p(x,\xi;h)\in S^0_\infty(1)$ be a real valued symbol such that $p(x,\xi;h)=p_0(x,\xi)+S^0(h)$  and let $P_h=\Op(p)$.
Let $\phi$ satisfy Hypotheses \ref{h1} and \ref{h2} and assume that the following assumptions hold true:
\begin{enumerate}[i)]
\item $P_h(e^{-\phi/h})=0$,
\item for all $x\in\R^d$, the function $\xi\in\R^d\mapsto p(x,\xi;h)$ is even,
  \item $\forall\delta>0,\ \exists\alpha>0, \ \forall (x,\xi)\in T^*\R^d, \quad (d(x,\uuu)^2+|\xi|^2\geq\delta\Longrightarrow p_0(x,\xi)\geq\alpha)$,
 \item for  any critical point $\u \in \uuu$ we have
\begin{equation*}
p_0(x,\xi)=|\xi|^2+|\nabla\phi(x)|^2+r(x,\xi),
\end{equation*}
with $r(x,\xi)=\ooo(|(x-\u ,\xi)|^3)$ near $(\u,0)$.
\end{enumerate}
Then, for $h>0$ small enough, there exists  symbol $q\in S^0(\D\aaa)$ satisfying the following properties.

First 
 $P_h=d_{\phi,h}^*Q^*Qd_{\phi,h}$ with 
$Q=\Op(q)$.
Next, $q(x,\xi;h)=q_0(x,\xi)+S^0(h\Xi\A)$ and
for any critical point  $\u \in \uuu$, we have
\begin{equation*}
q_0(x,\xi)=\Id+\ooo(\vert(x-\u ,\xi)\vert).
\end{equation*}
If we assume additionally that $r(x,\xi)=\ooo(|(x-\u ,\xi)|^4)$, then $q_0(x,\xi)=\Id+\ooo(|(x-\u,\xi)|^2)$ near $(\u,0)$ for any critical point  $\u \in \uuu$.

Eventually, if $p\in S^0_{cl}(1)$ then $q\in S^0_{cl}(\D\aaa)$.

\end{theorem}

Let us now make some comments on the above theorem. 
As  already mentioned, we decided in this paper not to give results in the most general
case so that technical aspects do not hide the main ideas. Nevertheless, we would like to mention here some possible generalizations of the preceding result.

First, it should certainly be possible to use  more general order functions and to prove a factorization results for symbols in other classes (for instance $S^0(\<(x,\xi)\>^2)$.
 This should allow to see the supersymmetric structure of the Witten Laplacian as a special case of our result. In other words, the symbol  $p(x,\xi;h)=\vert \xi\vert^2+\vert\nabla\phi(x)\vert^2-h\Delta\phi(x)$ would satisfy Assumptions i) to iv) above.

The analyticity of the symbol $p$ with respect to variable $\xi$ is certainly not necessary in order to get a factorization result
(it suffices to take a non analytic $q$ in the conclusion to see it). Nevertheless, since our approach
consists in conjugating the operator by $e^{-\phi/h}$ it seems difficult to deal with non analytic symbols.
Moreover, using a regularization procedure in the proof the above theorem,  it is certainly possible to prove that the symbol
$q$ above can be chosen in a class $S^0_\tau(\D\aaa)$ for some $\tau>0$.
Using this additional property it may be possible to prove some Agmon estimates, construct more accurate quasimodes (on the $1$-forms) and then to prove a full asymptotic expansion in Theorem \ref{e3}.

A  more delicate question should be to get rid of the parity assumption ii). 
It is clear that this assumption is not necessary (take $q(x,\xi)=\<\xi\>^{-2}(\Id+\diag(\xi_i/\<\xi\>))$ in the conclusion) 
but it seems difficult to prove a factorization result without it. For instance, if we consider the case 
$\phi=0$ (which doesn't enter exactly in our framework but enlightens easily the situation) then $P_h=hD_{x_j}$ is of order $1$ and 
can not be smoothly factorized both left and right.

As it will be seen in the proof below, the operator $Q$ (as well as $Q^*Q$) above is not unique. Trying to characterize the set of all possible $Q$ should be also a question of 
interrest.

Eventually, optimality of assumption iv) should be questioned. Expanding $q_0$ near $(\u,0)$ we can see that we have necessarily
$$p_0(x,\xi)=\vert q_0(\u,0)(\xi-i\nabla\phi)\vert^2+\ooo(\vert(x-\u,\xi)\vert^3)$$
near any critical point.  In assumption iv) we consider the case $q_0(\u,0)=\Id$, but it could be easily relaxed to any invertible matrix $q_0(\u,0)$.

\medskip

The plan of the article is the following.  In the next section we analyse the structure of operator $\T_h^\star$
and prove the first results on the spectrum stated in Theorem \ref{e1}. In section 3 we prove Theorem \ref{e4} and apply it to the case of the random walk operator.
In section 4, we prove some preliminary spectral results, and in section 5, we prove Theorem \ref{e3}.

\section{Structure of the operator and first spectral results} \label{s2}

In this section, we analyse the structure of the spectrum of the operator $\T_h^\star$ on the space $\hhh_{h} = L^{2} ( d \nu_{h,\infty} )$ (see \eqref{a10}). But this is more convenient to work with the standard Lebesgue measure, than with the measure $d \nu_{h,\infty}$. We then introduce the Maxwellian $\mmm_h$ defined by
\begin{equation} \label{a52}
d \nu_{h , \infty} = \mmm_h(x) \, d x \qquad \text{so that} \qquad \mmm_h = \widetilde{Z}_h \mu_h ( B_{h} (x) ) Z_h e^{- \phi ( x ) / h} ,
\end{equation}
and we make the following change of function
\begin{equation*}
\U_h u(x) : = \mmm_h^{-1/2}(x) u (x) ,
\end{equation*}
where $\U_{h}$ is unitary from $L^2 ( \R^d ) = L^2 ( \R^d , d x )$ to $\hhh_h$. Denoting
\begin{equation} \label{a3}
T_h : = \U_h^* \T_h^\star \U_h ,
\end{equation}
the conjugated operator acting in $L^2(\R^d)$, we have
\begin{align*}
 T_h u(x) &  = Z_h  \mmm_h^{-1/2}(x) e^{-\phi(x)/h} \int_{\R^{d}}  \one_{\vert x-y\vert<h} \mmm_h^{1/2}(y)  \mu_h(B(y,h))^{-1} u(y) \, d y  \\
& = \bigg( \frac{ Z_h e^{-\phi(x)/h}}{\mu_h(B(x,h))} \bigg)^{1/2} \int_{\vert x-y\vert<h}u(y) \bigg( \frac{Z_h e^{-\phi(y)/h}}{\mu_h(B(y,h))} \bigg)^{1/2} dy.
\end{align*}
We pose for the following
\begin{equation*}
a_h ( x ) = ( \alpha_d h^d )^{1/2} \bigg( \frac{ Z_h e^{-\phi(x)/h}}{\mu_h(B(x,h))}      \bigg)^{1/2} ,
\end{equation*}
and define the operator $\G$  by
\begin{equation} \label{a4}
 \G u(x)=\frac 1{\alpha_d h^d}\int_{\vert x-y\vert<h}u(y)dy
\end{equation}
where $\alpha_d = \vol (B(0,1))$ denotes the euclidean volume of the unit ball, so that with these notations,  operator $T_h$
reads
\begin{equation} \label{a2}
T_h = a_h \G a_h ,
\end{equation}
i.e.
\begin{equation*}
T_h u ( x ) = a_h ( x ) \G( a_h u ) (x) .
\end{equation*}
We note that
\begin{equation}\label{eq:form_ah}
a_h^{-2}(x) =   \frac{ \mu_h(B(x,h)) e^{\phi(x)/h}}{\alpha_d h^dZ_h} =
 \frac{1}{\alpha_d h^d} \int_{|x-y| <h}  e^{(\phi(x)-\phi(y))/h} dy = e^{\phi(x)/h} \G (e^{-\phi/h})(x) .
\end{equation}
We now collect some properties on $\G$ and $a_{h}$.

One very simple but fundamental observation is that $\G$ is a semiclassical Fourier multiplier $\G = G ( h D )=\Op(G)$ where
\begin{equation} \label{a8}
\forall \xi \in \R^{d} , \qquad G ( \xi ) = \frac{1}{\alpha_{d}} \int_{\vert z \vert < 1} e^{i z \cdot \xi} d z .
\end{equation}

\begin{lemma}\sl \label{a7}
The function $G$ is analytic on $\C^d$ and enjoys the following properties:
\begin{enumerate}[i)]
\item $G : \R^{d} \longrightarrow \R$.
\item There exists $\delta > 0$ such that $G ( \R^{d} ) \subset [ - 1 + \delta , 1 ]$. Near $\xi=0$, we have
\begin{equation*}
G ( \xi ) = 1 - \beta_{d} |\xi|^{2} + \ooo ( |\xi|^{4} ) ,
\end{equation*}
where $\beta_{d} = ( 2 d + 4)^{- 1}$. For any $r>0$, $\sup_{\vert \xi \vert \geq r} \vert G ( \xi )\vert < 1$ and $\lim_{\vert \xi \vert \rightarrow \infty} G ( \xi ) = 0$.
\item For all $\tau \in \R^{d}$, $G ( i \tau ) \in \R$, $G ( i \tau )\geq 1$ and, for any $r>0$, $\inf_{\vert \tau \vert \geq r} G ( i \tau ) > 1$.   \label{a12}
\item For all $\xi , \tau \in \R^{d}$ we have $\vert G ( \xi + i \tau ) \vert \leq G( i \tau )$.
\end{enumerate}
\end{lemma}

\begin{proof}
The function $G$ is analytic on $\C^{d}$ since it is the Fourier transform of a compactly supported distribution. The fact that $G ( \R^{d} )\subset \R$ is clear using the change of variable $z \mapsto - z$. The second item was shown in \cite{LeMi10_01}.

We now prove \textsl{iii)}. The fact that $G ( i \tau )$ is real for any $\tau \in \R^{d}$ is clear. Moreover, one can see easily that $\tau \mapsto G ( i \tau )$ is radial, so that there exists a function $\Gamma : \R \rightarrow \R$ such that, for all $\tau \in \R^{d}$, $G ( i \tau ) = \Gamma ( \vert \tau \vert )$. Simple computations show that $\Gamma$ enjoys the following properties:
\begin{itemize}
\item $\Gamma$ is even,
\item $\Gamma$ is strictly increasing on $\R_+$,
\item $\Gamma(0)=1$.
\end{itemize}
This leads directly to the announced properties on $G ( i \tau )$.

Finally, the fact that  for all $\xi , \tau \in \R^{d}$ we have $\vert G ( \xi + i \tau ) \vert \leq G ( i \tau )$ is trivial since, for all $z \in \R^{d}$, $\vert e^{i z \cdot ( \xi + i \tau ) }\vert = e^{- z \cdot \tau}$.
\end{proof}

\begin{lemma}\sl \label{a5}
There exist $c_1 , c_2 > 0$ such that $c_1 < a_h (x) < c_2$ for all $x \in \R^d$ and $h \in ] 0 , 1 ]$. Moreover, the functions $a_h$ and $a_h^{-2}$ belong to $S^{0} (1)$ and have classical expansions $a_{h} = a_{0} + h a_{1} + \cdots$ and $a_{h}^{-2} = a_{0}^{-2} + \cdots$. In addition,
\begin{align*}
a_{0} (x) &= G ( i \nabla \phi ( x ) )^{- 1 / 2} ,  \\
a_{1} (x) &= G ( i \nabla \phi ( x ) )^{- 3 / 2} \frac{1}{4 \alpha_{d}} \int_{\vert z \vert < 1} e^{- \nabla \phi (x) \cdot z} \big\< \phi'' (x) z , z \big\> \, d z .
\end{align*}
Eventually, there exist $c_0 , R > 0$ such that for all $\vert x \vert \geq R$, $a_h^{-2} ( x ) \geq 1 + c_0$ for $h>0$ small enough.
\end{lemma}

\begin{proof}
By a simple change of variable, we have
\begin{equation*}
a_h^{-2} ( x ) =\frac{1}{\alpha_d} \int_{\vert z \vert < 1} e^{( \phi ( x ) - \phi ( x + h z ) ) / h} d z .
\end{equation*}
 Since there exists $C>0$ such that $\vert\nabla\phi(x)\vert\leq C$ for all $x\in\R^d$, then we can find some constants $c_1,c_2>0$ such that $c_1 < a_h (x)^{-2} < c_2$ for all $x \in \R^d$ and $h \in ]0,1]$.  Moreover, thanks to the bounds on the derivatives of $\phi$, we get easily that derivatives of $a_h^{-2}$ are also bounded. This shows that $a_h^{-2}$ belongs to $S^0 (1)$ and, since it is bounded from below by $c_1>0$, we get immediately that $a_h \in S^{0} (1)$.

On the other hand, by simple Taylor expansion, we get that $a_{h}$ and $a_{h}^{-2}$ have classical expansions and the required expressions for $a_{0}$ and $a_{1}$. Since $\vert \nabla \phi ( x ) \vert \geq c > 0$ for $x$ large enough, it follows from Proposition \ref{a7} \textsl{iii)} that there exists $c_0 , R>0$ such that for all $\vert x \vert \geq R$, $G ( i \nabla \phi ( x ) ) \geq 1 + 2 c_0$, and hence $a_h^{-2} ( x ) \geq 1 + c_0$, for $h > 0$ sufficiently  small.
\end{proof}

Since we want to study the spectrum near $1$, it will be convenient to introduce
\begin{equation} \label{a49}
P_h : = 1 - T_{h} .
\end{equation}
Using \eqref{a2} and \eqref{eq:form_ah}, we get
\be\label{eq:form_Ph}
P_h=a_h(V_h(x)-G(hD_x))a_h
\ee
with $V_h(x)=a_h^{-2}(x)=e^{\phi/h}G(hD_x)(e^{-\phi/h})$.
As a consequence of the previous lemmas, we get the following proposition for $P_h$.

\begin{proposition}\sl \label{a11}
The operator $P_{h}$ is a semiclassical pseudodifferential operator whose symbol $p (x , \xi ;h ) \in S^{0}_{\infty} (1)$ admits a classical expansion which reads $p = p_{0} + h p_{1} + \cdots$ with
\begin{align*}
p_{0} ( x , \xi ) = 1 - G ( i \nabla \phi (x) )^{- 1} G ( \xi ) \geq 0  \qquad \text{and} \qquad p_{1} ( x , \xi ) = G_{1} ( x ) G ( \xi ) ,
\end{align*}
where
\begin{align*}
G_{1} (x) &= - G ( i \nabla \phi (x) )^{- 2} \frac{1}{2 \alpha_{d}} \int_{\vert z \vert < 1} e^{- \nabla \phi (x) \cdot z} \big\< \phi'' (x) z , z \big\> \, d z \\
& = - \beta_{d} \Delta \phi ( \u ) + \ooo ( \vert x - \u \vert ) ,
\end{align*}
near any $\u \in \uuu$.
\end{proposition}

\begin{proof}
The fact that $p$ belongs to $S^{0}_{\infty} (1)$ and admits a classical expansion is clear thanks to Lemma \ref{a7} and Lemma \ref{a5}. From the standard pseudodifferential calculus, the symbol $p$ satisfies
\begin{align*}
p ( x , \xi ; h ) &= 1 - a_{0}^{2} G - 2 a_{0} a_{1} G h - \frac{h}{2 i} a_{0} \{ G , a_{0} \} - \frac{h}{2 i} \{ a_{0} , a_{0} G \} + S^{0} ( h^{2} )  \\
&=  1 - a_{0}^{2} G - 2 a_{0} a_{1} G h + S^{0} ( h^{2} ) .
\end{align*}
Combined with Lemma \ref{a5}, this leads to the required expressions of $p_{0}$ and $p_{1}$.

Finally, the non-negativity of $p_{0}$ comes from the formula
\begin{equation*}
p_{0} = G ( i \nabla \phi (x) )^{- 1} \big( ( 1 - G ( \xi ) ) + ( G ( i \nabla \phi (x) ) - 1 ) \big) ,
\end{equation*}
and Lemma \ref{a7} which implies $1 - G ( \xi ) \geq 0$ and $G ( i \nabla \phi (x) ) - 1 \geq 0$.
\end{proof}

We finish this subsection with the following proposition which is a part of Theorem \ref{a1}.

\begin{proposition}\sl \label{a47}
There exist $\delta , h_{0} > 0$ such that the following assertions hold true for $h \in ] 0 , h_{0} ]$. First, $\sigma (T_h ) \subset [ - 1 + \delta , 1 ]$ and $\sigma_{ess} ( T_h ) \subset [ - 1 + \delta , 1 - \delta ]$. Eventually, $1$ is a simple eigenvalue for the eigenfunction $\mmm_h^{1/2}$.
\end{proposition}

\begin{proof}

We start by proving $\sigma ( T_{h} ) \subset [ - 1 + \delta , 1 ]$. From \eqref{a9}, we already know that $\sigma ( T_{h} ) \subset [ - 1 , 1 ]$. Moreover, Lemma \ref{a7} \textsl{ii)} and \textsl{iii)} imply $0 \leq a_{0} (x) \leq 1$ and $G ( \R^{d} ) \subset [ - 1 + \nu , 1 ]$ for some $\nu > 0$. Thus, we deduce that the symbol $\tau_{h} ( x , \xi )$ of the
pseudodifferential operator $T_{h} \in \Psi^{0} (1)$ satisfies
\begin{equation*}
\tau_{h} ( x , \xi ) \geq - 1 + \nu + \ooo (h) .
\end{equation*}
Then, G{\aa}rding's inequality yields
\begin{equation*}
T_{h} \geq - 1 + \nu /2 ,
\end{equation*}
for $h$ small enough. Summing up, we obtain $\sigma ( T_{h} ) \subset [ - 1 + \delta , 1 ]$.

Let us prove the assertion about the essential spectrum. Let $\chi \in C_{0}^{\infty} ( \R^{d} ; [0 , 1 ] )$ be equal to $1$ on $B ( 0 ,R )$, where $R > 0$ is as in Lemma \ref{a5}. Since $\G = G ( h D ) \in \Psi^{0} (1)$ and $\lim_{\vert \xi \vert \rightarrow \infty} G ( \xi ) = 0$, the operator
\begin{equation*}
T_{h} - (1 - \chi ) T_h (1 - \chi ) = \chi T_{h} + T_{h} \chi - \chi T_{h} \chi ,
\end{equation*}
is compact. Hence, $\sigma_{ess} ( T_{h} ) = \sigma_{ess} ( (1 - \chi ) T_h (1 - \chi ) )$. Now, for all $u \in L^2 ( \R^{d} )$, we have
\begin{align*}
\big\< ( 1 - \chi ) T_{h} ( 1 - \chi ) u , u \big\> &= \big\< \G a_h ( 1 - \chi ) u, a_h ( 1 - \chi ) u \big\> \\ &\leq \Vert a_h ( 1 - \chi ) u \Vert^{2} \leq (1 + c_{0} )^{- 1} \Vert u \Vert^{2} ,
\end{align*}
since $\Vert \G \Vert_{L^2 \rightarrow L^2} \leq 1$ and $\vert a_h ( 1 - \chi ) \vert \leq (1 + c_{0} )^{- 1 / 2}$ thanks to Lemma \ref{a7} \textsl{ii)} and Lemma \ref{a5}. As a consequence, there exists $\delta >0$ such that $\sigma_{ess} ( T_h ) \subset [ - 1 + \delta , 1 - \delta ]$.

To finish the proof, it remains to show that $1$ is a simple eigenvalue. Let $k_{h} ( x , y )$ denotes the distribution kernel of $T_{h}$. From \eqref{a4}, \eqref{a2} and Lemma \ref{a5}, there exists $\varepsilon > 0$ such that, for all $x , y \in \R^{d}$,
\begin{equation} \label{a6}
k_{h} (x , y ) \geq \varepsilon h^{-d} \one_{\vert x - y \vert < h} .
\end{equation}
We now consider $\widetilde{T}_{h} = T_{h} + 1$. Since $\Vert T_{h} \Vert =1$, the operator $\widetilde{T}_{h}$ is bounded and non-negative. Moreover, $\mmm_h^{1/2}$ is clearly an eigenvector associated to the eigenvalue $\Vert \widetilde{T}_{h} \Vert = 2$. On the other hand, \eqref{a6} implies that $\widetilde{T}_{h}$ is positivity preserving (this means that $u (x) \geq 0$ almost everywhere and $u \neq 0$ implies $\widetilde{T}_{h} u (x) \geq 0$ almost everywhere and $\widetilde{T}_{h} u \neq 0$). Furthermore, $\widetilde{T}_{h}$ is ergodic (in the sense that, for any $u , v \in L^{2} ( \R^{d} )$ non-negative almost everywhere and not the zero function, there exists $n \geq 1$ such that $\< u , \widetilde{T}^{n}_{h} v \> > 0$). Indeed, let $u,v$ be two such functions. We have $\< u , \widetilde{T}^{n}_{h} v \> \geq \< u , T^{n}_{h} v \>$ where, from \eqref{a6}, the distribution kernel of $T_{h}^{n}$ satisfies
\begin{equation*}
k_{h}^{(n)} (x , y ) \geq \varepsilon_{n} h^{-d} \one_{\vert x - y \vert < n h} ,
\end{equation*}
with $\varepsilon_{n} > 0$. Thus, if $n \geq 1$ is chosen such that $\dist ( \esssupp ( u ) , \esssupp ( v ) ) < n h$, we have $\< u , \widetilde{T}^{n}_{h} v \> > 0$.
Eventually, the above properties of $\widetilde{T}_{h}$ and the Perron--Frobenius theorem (see Theorem XIII.43 of
\cite{ReSi78_01}) implies that $1$ is a simple eigenvalue of $T_{h}$.
\end{proof}

\section{Supersymmetric structure}
In this section, we prove that the operator $\Id-\T^\star_h$ admits a  supersymmetric structure and prove Theorem \ref{e4}. We showed in the preceding section that
$$
\Id- \T^\star_h = \U P_h \U^*
$$
and
before proving Theorem \ref{e4}, we state and prove as a corollary the main result
on the operator $P_h$. Recall here that $\beta_{d} = ( 2 d + 4)^{- 1}$ and $\D\A$ is the matrix
symbol defined by $\D\A_{i,j} = \< \xi_j \>^{-1}$, for all $i , j = 1, \ldots , d$.

\begin{corollary}\sl \label{e8}
There exists a classical symbol $q\in S^0_{cl}( \D\A)$ such that the following holds true.
 First $P_h=L_\phi^*L_\phi$ with $L_\phi=Qd_{\phi,h} a_h $ and
$Q=\Op(q)$.
Next, 
$q=q_0+\Psi^0(h\Xi\A)$  with $q_0(x,\xi)=\beta_d^{1/2}\Id+\ooo(|(x-\u ,\xi)|^2)$ for any critical point $\u \in \uuu$.
\end{corollary}

\begin{proof}
Since we know that $P_h=  a_h ( V_h(x) - \G) a_h $, we only have to prove
that $\beta_d^{-1} \widetilde{P}_h$ satisfies the asumptions of Theorem \ref{e4}, where
\be
\widetilde{P}_h=V_h(x)-G(hD).
\ee

Assumption i) is satisfied by construction.

Observe that thanks to Proposition \ref{a11}, it is a pseudodifferential operator and since variable $x$ and $\xi$ are separated, its symbol in any quantization is given by  $\tilde{p}_h(x,\xi)=V_h(x)-G(\xi)$. Moreover, Lemma \ref{a5} and   Proposition \ref{a11} show that
$\tilde{p}_h$ admits a classical expansion $\tilde{p}=\sum_{j=0}^\infty h^j\tilde{p}_j$ with $\tilde{p}_j$, $j\geq 1$ depending only on $x$ and $\tilde{p}_0(x,\xi)=G(i\nabla\phi(x))-G(\xi)$.
Hence, it follows from Lemma \ref{a7} that $\tilde{p}$ satisfies assumptions ii) and iii).

Finally, it follows from ii) of Lemma \ref{a7} that near $( \u , 0 )$ (for any $\u \in \uuu$) we have
\begin{equation*}
\tilde{p}(x,\xi)=\beta_d(|\xi|^2+|\nabla\phi(x)|^2)+\ooo(|(x- \u ,\xi)|^4)+S^0(h) ,
\end{equation*}
so that we can apply Theorem \ref{e4} in the case where $r=\ooo(|(x- \u ,\xi)|^4)$ .
Taking into account the multiplication part $a_h$ completes the proof for $P_h$.
\end{proof}

Now we can do the

{\it Proof of Theorem \ref{e4}.} Given a symbol $p\in S^0(1)$ we recall first the well-known left and  right  quantization
\begin{equation}
\Op^l(p)u(x)=(2\pi h)^{-d}\int_{T^*\R^d} e^{ih^{-1}(x-y)\xi}p(x,\xi)u(y)dyd\xi ,
\end{equation}
and
\begin{equation}
\Op^r(p)u(x)=(2\pi h)^{-d}\int_{T^*\R^d} e^{ih^{-1}(x-y)\xi}p(y,\xi)u(y)dyd\xi ,
\end{equation}
Finally, if $p(x,y,\xi)$ belongs to $S^0(1)$, we define $\widetilde\Op(p)$, by
\begin{equation}
\widetilde\Op( p)(u)(x)=(2\pi h)^{-d}\int_{T^*\R^d} e^{ih^{-1}(x-y)\xi}p(x,y,\xi)u(y)dyd\xi,
\end{equation}
We recall the formula allowing to pass from one of these quantizations to the other. If $p(x,y,\xi)$ belongs to $S^0(1)$, then
$\widetilde\Op(p)=\Op^l(p_l)=\Op^r(p_r)$ with
\begin{equation} \label{e9}
p_l(x,\xi)=(2\pi h)^{-d}\int_{T^*\R^d}e^{ih^{-1}z(\xi'-\xi)}p(x,x-z,\xi')d\xi'dz ,
\end{equation}
and
\begin{equation} \label{e10}
p_r(y,\xi)=(2\pi h)^{-d}\int_{T^*\R^d}e^{ih^{-1}z(\xi'-\xi)}p(y+z,y,\xi')d\xi'dz .
\end{equation}
Recall that we introduced the $d$-matrix-weight, $\A:T^*\R^d\rightarrow \M_d$ given by $\A_{i,j}(x,\xi)=(\<\xi_i\>\<\xi_j\>)^{-1}$.
Suppose eventually that $p$ satisfies the hypotheses of Theorem \ref{e4}: $P=\Op( p)$ with $p\in S^0_\infty(1)$, $p(x,\xi;h)=p_0(x,\xi)+S^0(h)$  such that
\begin{enumerate}[i)]
\item $P(e^{-\phi/h})=0$,
\item For all $x\in\R^d$, the function $\xi\in\R^d\mapsto  p(x,\xi;h)$ is even,
  \item $\forall\delta>0,\ \exists\alpha>0, \ \forall (x,\xi)\in T^*\R^d, \quad (d(x,\uuu)^2+|\xi|^2\geq\delta\Longrightarrow p_0(x,\xi)\geq\alpha)$,
 \item near any critical points $\u \in \uuu$ we have
\begin{equation*}
p_0(x,\xi)=|\xi|^2+|\nabla\phi(x)|^2+r(x,\xi) ,
\end{equation*}
 with either $r=\ooo(|(x-\u ,\xi)|^3)$ (assumption (A2)), or $r=\ooo(|(x-\u ,\xi)|^4)$ (assumption (A2')).
 \end{enumerate}
Observe here that the symbol $p$ may depend on $h$ but we omit this dependance in order to enlight the notations.

The proof goes in several steps. First we prove that there exists a symbol $\hat q\in S^0_\infty(\A)$ such that
\bes
P_h=\Dhs\widehat Q \Dh
\ees
where $\widehat Q=\Op(\hat q)$.

In a  second time we shall prove that the operator $\widehat Q$ can be chosen so that $\widehat Q=Q^*Q$ for some
pseudodifferential operator  $Q$ satisfying some good properties.

Let us start with the first step. For this purpose we need the following lemma
\begin{lemma}\sl \label{e35}
 Let $p\in S^0_\infty(1)$ and $P_h=\Op(p)$.  Assume that for all $x\in\R$, the function
 $\xi\mapsto  p(x,\xi;h)$ is  even. Suppose also that $P_h(e^{-\phi/h})=0$. Then there exists $\hat q\in  S^0_\infty(\A)$ such that
 $P_h=\Dh^*\widehat Q\Dh$ with $\widehat Q=\Op(\hat q)$. Moreover, if
 $p$ has a principal symbol, then so does $\hat q$ and if $p\in S^0_{\infty,cl}$ then $\hat q\in S^0_{\infty,cl}$ .
 \end{lemma}

\begin{remark}\sl
Since $P_h(e^{-\phi/h})=0$, it is quite clear that $P_h$ can be factorized by $d_{\phi,h}$ on the right. On the other hand, the fact that $P_h$ can be factorized by 
$d_{\phi,h}^*$ on the left necessarily implies that $P_h^*(e^{-\phi/h})=0$. At a first glance,  there is no reason for this identity to hold true since we don't suppose in the above lemma that $P_h$ is self-adjoint. This is actually verified for the following reason. Start from $\Op( p)(e^{-\phi/h})=0$, then taking the conjugate and using the fact that $\phi$ is real we get
$$\Op(\overline p(x,-\xi))(e^{-\phi/h})=0.$$
Hence, the parity assumption on $p$ implies that $\Op( p)^*(e^{-\phi/h})=0.$
\end{remark}
Let us now give the proof of the lemma.

\begin{proof}
The fondamental very simple remark is that if $a$ is a symbol such that $a(x,\xi)=b(x,\xi)\cdot\xi$ then the operator
$\Op^l(a)$ can be factorized by $hD_x$ on the right: $\Op^l(a)=\Op^l(b)\cdot hD_x$ whereas the right-quantization of $a$ can be factorized on the left:
$\Op^r(a)=hD_x\cdot \Op^r(b)$. We have to implement this simple idea, dealing with the fact that our operator are twisted by $e^{\phi/h}$.

Introduce the operator $P_{\phi,h}=e^{\phi/h}P_he^{-\phi/h}$.
Then, for any $u\in\sss(\R^d)$
\begin{equation*}
P_{\phi,h} u(x)=(2\pi h)^{-d}\int e^{ih^{-1}(x-y)\xi}e^{h^{-1}(\phi(x)-\phi(y))}p \Big( \frac{x+y}2,\xi \Big) u(y)dyd\xi.
\end{equation*}
We now use the Kuranishi trick. Let $\theta(x,y)=\theta_h(x,y)\in\R^d$ be defined by $\phi(x)-\phi(y)=(x-y)\cdot \theta(x,y)$.
Then
\begin{equation*}
P_{\phi,h} u(x)=(2\pi h)^{-d}\int e^{ih^{-1}(x-y)(\xi-i\theta(x,y))} p \Big( \frac{x+y}2,\xi \Big) u(y)dyd\xi
\end{equation*}
and since $p\in S^0_\infty$, a simple change of integration path shows that
 $P_{\phi,h}$ is a bounded pseudodifferential operator $P_{\phi,h}=\widetilde \Op(\tilde p_\phi)$
with
\begin{equation*}
\tilde p_\phi(x,y,\xi)=p \Big( \frac{x+y}2,\xi+i\theta(x,y) \Big) .
\end{equation*}
To get the expression of $P_{\phi,h}$ in left quantization, it suffice then to apply formula \eqref{e9} to get
 $P_{\phi,h}=\Op^l(p_\phi)$ with
\begin{align*}
p_\phi(x,\xi)&=(2\pi h)^{-d}\int_{\R^{2d}}e^{ih^{-1}(\xi'-\xi)(x-z)}p \Big( \frac{x+z}2,\xi'+i\theta(x,z) \Big) d\xi'dz\\
 &=(2\pi h)^{-d}\int_{\R^{2d}}e^{ih^{-1}\xi'(x-z)}p \Big( \frac{x+z}2,\xi'+\xi+i\theta(x,z) \Big) d\xi'dz.
\end{align*}
Observe that for any smooth function $g:\R^d\rightarrow \R$ we have
\begin{equation} \label{e11}
 g(\xi)-g(0)=\sum_{j=1}^d\int_0^1\xi_j\partial_{\xi_j}g(\gamma_j^\pm(s,\xi))ds,
\end{equation}
 with $\gamma_j^+(s,\xi)=(\xi_1,\dots,\xi_{j-1},s\xi_j,0,\ldots,0)$ and
 $\gamma_j^-(s,\xi)=(0,\ldots,0,s\xi_j,\xi_{j+1},\ldots,\xi_d)$.
 A very simple observation is that for any $(x,\xi)\in T^*\R^d$ and any $s\in [0,1]$ we have
$x\cdot\gamma_j^\pm(s,\xi)=\gamma_j^\pm(s,x)\cdot\xi$. This will be used often in the sequel.

 Let us go back to the study of $p_\phi$. Since $P_h(e^{-\phi/h})=0$, then  $p_\phi(x,0)=0$ and  by \eqref{e11}, we get
\begin{equation*}
p_\phi(x,\xi)=\sum_{j=1}^d\xi_j\check q^\pm_{\phi,j}(x,\xi)=\sum_{j=1}^d\xi_j\check q_{\phi,j}(x,\xi) ,
\end{equation*}
with $\check q_{\phi,j}=(\check q^+_{\phi,j}+\check q^-_{\phi,j})/2$ and
\begin{equation*}
\check q_{\phi,j}^\pm(x,\xi)=(2\pi h)^{-d}\int_{\R^{2d}}e^{ih^{-1}\xi'(x-z)}\int_0^1\partial_{\xi_j} p \Big( \frac{x+z}2,\xi'+\gamma_j^\pm(s,\xi)+i\theta(x,z) \Big) dsdzd\xi',
\end{equation*}
where the above integral has to be understood as an oscillatory integral. Since $\partial_\xi^\alpha p$ is bounded for any $\alpha$,
then integration by parts with respects to $\xi'$ and $z$ show that $q^\pm_{\phi,j}\in S_\infty^0(1)$. Moreover, by definition of $\gamma_j^\pm$, we have
\begin{equation*}
\xi_jq^\pm_{\phi,j}=(2\pi h)^{-d}\int_{\R^{2d}}e^{ih^{-1}\xi'(x-z)}c^\pm_j(x,z,\xi)dzd\xi ,
\end{equation*}
with
$c^\pm_j(x,z,\xi)=p(\frac{x+z}2,\xi'+\gamma_j^\pm(1,\xi)+i\theta(x,z))-p(\frac{x+z}2,\xi'+\gamma_j^\pm(0,\xi)+i\theta(x,z))$. This symbol is clearly in $S^0_\infty(1)$ so that
integration by parts as before show that $\xi_jq^\pm_{\phi,j}\in S^0_\infty(1)$. Since $\xi_j$ and $ q^\pm_{\phi,j}$ are both scalar, this proves that $q^\pm_{\phi,j}\in S^0_\infty(\<\xi_j\>^{-1})$.

Observe now that
\begin{equation*}
P_h=e^{-\phi/h}P_{\phi,h} e^{\phi/h}=e^{-\phi/h}\Op^l \Big( \frac{\check q^+_\phi+\check q^-_\phi}2 \Big) \cdot \Big( \frac h i\nabla_x \Big) e^{\phi/h}
=e^{-\phi/h}\widetilde Qe^{\phi/h}\cdot\Dh ,
\end{equation*}
with $\widetilde Q=\frac 1 {2i}(\widetilde Q^++\widetilde Q^-)$ and $\widetilde Q^\pm=\Op^l(\check q^\pm_\phi)$.
Let $\widetilde Q^\pm_\phi=e^{-2\phi/h}\Op^l(\check q^\pm_\phi)e^{2\phi/h}$, then $\widetilde Q^\pm_\phi=\widetilde\Op(\tilde q^\pm_\phi)$ with
$\tilde q^\pm_\phi=(\tilde q^\pm_{\phi,1},\ldots,\tilde q^\pm_{\phi,d})$ and
\begin{equation*}
\begin{split}
\tilde q^\pm_{\phi,j}(x,y,\xi)&=\check q^\pm_{\phi,j}(x,\xi-2i\theta(x,y))
=(2\pi h)^{-d}\int_{\R^{2d}}e^{ih^{-1}\xi'(x-z)}\\
&\int_0^1 \partial_{\xi_j} p \Big( \frac{x+z}2,\xi'+\gamma_j^\pm(s,\xi)-2i\gamma_j^\pm(s,\theta(x,y))+i\theta(x,z) \Big) dsdzd\xi' ,
\end{split}
\end{equation*}
and it follows from \eqref{e10} that $\widetilde Q_\phi=\Op^r( \breve{q}_\phi)$ with
$\breve{q}_\phi=\breve{q}^+_\phi+\breve{q}^-_\phi$, $\breve q^\pm_\phi=( \breve q^\pm_{\phi,1},\ldots, \breve q^\pm_{\phi,d})$ and
\begin{align*}
\breve{q}^\pm_{\phi,j} (x,\xi) &=(2\pi h)^{-d}\int_{\R^{2d}}e^{ih^{-1}(\xi'-\xi)u}\tilde q^\pm_{\phi,j}(x+u,x,\xi')dud\xi'\\
&=(2\pi h)^{-2d}\int_{\R^{4d}}\int_0^1e^{ih^{-1}[(\xi'-\xi)u+(x+u-z)\eta]}\\
\partial_{\xi_j} p & \Big( \frac {x+u+z}2,\eta+\gamma_j^\pm(s,\xi')-2i\gamma_j^\pm(s,\theta(x+u,x))+i\theta(x+u,z) \Big)ds dzdud\xi'd\eta.
\end{align*}
Make the change of variables $z=x+v$ and $\nu=\gamma_j^\pm(s,\xi')+\eta$, the above equation yields
\begin{align*}
\breve{q}^\pm_{\phi,j}(x,\xi)=(2\pi h)^{-2d}\int_{\R^{4d}}\int_0^1&e^{ih^{-1}[(\xi'-\xi)u +(u-v)(\nu-\gamma_j^\pm(s,\xi'))]}\\
&\phantom{***}\partial_{\xi_j} p \Big( x+\frac{u+v}2,\nu+\psi_j^\pm(s,x,u,v) \Big)ds du dv d\nu d\xi' ,
\end{align*}
with
$\psi_j^\pm(s,x,u,v)=i\theta(x+u,x+v)-2i\gamma_j^\pm(s,\theta(x+u,x))$.

Denote $\hat p^2(x,z)=\int e^{-iz\xi}p(x,\xi)d\xi$ the Fourier transform of $p$ with respect to the second variable and observe that since
$\xi\mapsto p(x,\xi)$ is even, then so is $z\mapsto \hat p^2(x,z)$.
Using the above notations, we have $$\partial_{\eta_j} p(x,\eta)=i\int_{\R^d}e^{iz\eta}z_j \hat p^2(x,z)dz,$$
and we get
\begin{align*}
\breve{q}^\pm_{\phi,j}(x,\xi)=i(2\pi h)^{-2d}\int_{\R^{4d}\times[0,1]\times\R^d}&ze^{ih^{-1}[(u-v+hz)\nu+(\xi'-\xi)u-(u-v)\gamma_j^\pm(s,\xi')]}\\
&\hat p^2 \Big( x+\frac{u+v}2,z \Big) e^{iz\psi_j^\pm(s,x,u,v)}dudvd\xi'd\nu ds dz.
\end{align*}

Let $\F_{h,\nu\mapsto v}$ denote the semiclassical Fourier transform with respect to variable $\nu$ and $\overline\F_{h,u\mapsto \nu}$ its inverse.
Denoting
\begin{equation*}
f_{s,x,v,z}(u)=z\hat p^2 \Big( x+\frac{u+v}2,z \Big) e^{ih^{-1}[  (\xi'-\xi)u-(u-v)\gamma_j^\pm(s,\xi') ]}e^{iz\psi_j^\pm(s,x,u,v)} ,
\end{equation*}
we get
\begin{align*}
\breve{q}^\pm_{\phi,j}(x,\xi)&=i(2\pi h)^{-d}\int_{\R^{2d}\times[0,1]\times\R^d}\F_{h,\nu\mapsto v}\overline\F_{h,u\mapsto\nu}(f_{s,x,v,z})(v-hz) dvd\xi' ds dz\\
&=i(2\pi h)^{-d}\int_{\R^{2d}\times[0,1]\times\R^d} ze^{ih^{-1}[ (\xi'-\xi)(v-hz)+h\gamma_j^\pm(s,z)\xi'    ]}\\
&\phantom{************}\hat p^2(x+v-hz/2,z)e^{i\psi_j^\pm(s,x,v-hz,v)z}dvd\xi' ds dz ,
\end{align*}
where we have used the fact that $\gamma_j^\pm(s,\xi')z  =\gamma_j^\pm(s,z)\xi' $.
Similarly, integrating with respect to $\xi'$ and $v$, we obtain
\begin{equation*}
\breve{q}^\pm_{\phi,j}(x,\xi)=i\int_{[0,1]\times\R^d}z_j e^{i\gamma_j^\pm(s,z)\xi}\hat p^2 \Big( x+h \big( \frac z 2-\gamma_j^\pm(s,z) \big) ,z \Big) e^{\varphi_j^\pm(s,z)}dsdz ,
\end{equation*}
with $\varphi_j^\pm(s,z)=iz\psi_j^\pm(s,x,-h\gamma_j^\pm(s,z),h(z-\gamma_j^\pm(s,z))$.
From the definition of $\psi_j^\pm$, we get
\begin{align*}
\varphi_j^\pm(s,z)&=2z\gamma_j^\pm(s,\theta(x-h\gamma_j^\pm(s,z),x))-z\theta(x-h\gamma_j^\pm(s,z),x+h(z-\gamma_j^\pm(s,z)))\\
&=2\gamma_j^\pm(s,z)\theta(x-h\gamma_j^\pm(s,z),x)-z\theta(x-h\gamma_j^\pm(s,z),x+h(z-\gamma_j^\pm(s,z))) ,
\end{align*}
and since $\theta$ is defined by $\phi(x)-\phi(y)=(x-y)\theta(x,y)$, it follows easily that
\begin{equation*}
\varphi_j^\pm(s,z)=\frac 1{h}(2\phi(x)-\phi(x-h\gamma_j^\pm(s,z))-\phi(x+h(z-\gamma_j^\pm(s,z))).
\end{equation*}
Let us denote $\rho_j^\pm(x,s,z)=\hat p^2(x+h(\frac z 2-\gamma_j^\pm(s,z)),z)$, then
\begin{equation}\label{e12}
\breve{q}^\pm_{\phi,j}(x,0)=i\int_{[0,1]\times\R^d}z_j \rho_j^\pm(x,s,z)e^{\varphi_j^\pm(s,z)}dsdz.
\end{equation}
Observe now that we have the following identities
\be
\begin{split}
\gamma_j^\pm(1-s,-z)&=-(z-\gamma_j^\mp(s,z)) ,  \\
\frac z 2-\gamma_j^\pm(s,z)&=\frac {-z} 2-\gamma_j^\mp(1-s,-z) ,
\end{split}
\ee
for all $s\in[0,1],z\in\R^d$. In particular, since $\hat p^2$ is even with respect to the second variable, we get
\begin{equation*}
\rho_j^\pm(x,1-s,-z)e^{\varphi_j^\pm(1-s,-z)}=\rho_j^\mp(x,s,z)e^{\varphi_j^\mp(s,z)}.
\end{equation*}
As a consequence, using the change of variable $(s,z)\mapsto (1-s,-z)$ in \eqref{e12}, we get
$\breve{q}^+_{\phi,j}(x,0)=-\breve{q}^-_{\phi,j}(x,0)$ and hence
$\breve q_\phi(x,0)=0$.
Since $\breve q_{\phi,j}$ belongs to $S^0_\infty(\<\xi_j\>^{-1})$ for all $j$, we get by using the same trick as for the right-factorization that there exists
some symbol $q=(\overline q_{j,k})\in S^0_\infty(\A)$ such that
$\breve q_{\phi,j}(x,\xi)=\sum_{k=1}^d\xi_k\overline q_{j,k}(x,\xi)$.
Since we use right-quantization, it follows that for all $u\in \mathcal{S}(\R^d,\C^d)$,
\begin{equation*}
\Op^r(\breve{q}_\phi)u=\frac h i \operatorname{div}\Op^r(\overline q)u=hD_x^*\Op^r(\overline q)u ,
\end{equation*}
where we have used the matrix-valued symbol $q=(q_{j,k})$.
Consequently, for all $u\in\mathcal{S}(\R^d)$,
\begin{equation*}
P_hu=e^{\phi/h}\Op^r(\breve{q}_\phi)e^{-\phi/h}\Dh u=\Dhs e^{\phi/h}\Op^r(\overline q)e^{-\phi/h}\Dh u.
\end{equation*}
Using again analyticity of $\overline q$, there exists $\hat q\in S^0_\infty(A)$ such that
\begin{equation*}
\widehat{Q} : = e^{\phi / h} \Op^r ( \overline{q} ) e^{- \phi / h} = \Op ( \hat{q} ) ,
\end{equation*}
and the factorization  is proved. The fact that $\hat q$ admits an expansion in powers of $h$ follows easily from the above computations, since it is the case for $p$.
\end{proof}

Let us apply Lemma \ref{e35} to $P_{h} = \Op ( p )$. Then, there exists a symbol $\widehat{q} \in S_{\infty}^{0} ( \A )$ such that
\begin{equation*}
P_{h} = d_{\phi , h}^{*} \widehat{Q} d_{\phi , h} ,
\end{equation*}
with $\widehat{Q} = \Op ( \widehat{q} )$ and $\hat q=\hat q_0+S^0(h)$. Now the strategy is the following. We will modify the operator $\widehat{Q}$ so that the new $\widehat{Q}$ is selfadjoint, non-negative and $\widehat{Q}$ can be written as the square of a pseudodifferential operator $\widehat{Q} = Q^{*} Q$.

First observe that since $P_{h}$ is selfadjoint,
\begin{equation*}
P_{h} = \frac{1}{2} ( P_{h} + P_{h}^{*} ) = d_{\phi , h}^{*} \frac{\widehat{Q} + \widehat{Q}^*}{2} d_{\phi , h} ,
\end{equation*}
so that we can assume in the following that $\widehat{Q}$ is selfadjoint. This means that the partial operators $\widehat{Q}_{j , k} = \Op ( \widehat{q}_{j , k} )$ verify $\widehat{Q}_{j , k}^{*} = \widehat{Q}_{k , j}$ (or at the level of symbols $\hat{q}_{k , j} = \overline{\hat{q}_{j , k}} )$. For $k = 1 , \ldots , d$, let us denote $d_{\phi , h}^{k} = h \partial_{k} + \partial_{k} \phi (x)$. Then
\begin{equation} \label{e13}
P_{h} = \sum_{j , k = 1}^{d} ( d_{\phi , h}^{j} )^{*} \widehat{Q}_{j , k} d_{\phi , h}^{k} .
\end{equation}

We would like to take the square root of $\widehat{Q}$ and show that it is still a pseudodifferential operator. The problem is that we don't even know if $\widehat{Q}$ is non-negative. Nevertheless, we can use the non-uniqueness of operators $\widehat{Q}$ such that \eqref{e13} holds to go to a situation where $\widehat{Q}$ is close to a diagonal operator with non-negative partial operators on the diagonal. The starting point of this strategy is the following commutation relation
\begin{equation}
\forall j , k \in \{ 1 , \ldots , d \} , \qquad \big[ d_{\phi , h}^{j} , d_{\phi , h}^{k} \big] = 0 ,
\end{equation}
which holds true since $d_{\phi , h}^{j} = e^{- \phi / h} h \partial_{j} e^{\phi / h}$ and thanks to Schwarz Theorem. Hence, for any bounded operator $B$, we have
\begin{equation} \label{e33}
P_{h} = d_{\phi , h}^{*} \widehat{Q}^{mod , \bullet} d_{\phi , h} = \sum_{j , k = 1}^{d} ( d_{\phi , h}^{j} )^{*} \widehat{Q}^{mod , \bullet}_{j , k} d_{\phi , h}^{k} ,
\end{equation}
with $\widehat{Q}^{mod , \bullet} = \widehat{Q} + \B^{\bullet}$, $\bullet \in \{ 0 , \infty \}$, for some $\B^{\bullet}$ having one of the two following forms:
\begin{itemize}
\item[$\bullet$] Exchange between three coefficients. For any $j_{0} , k_{0} , n \in \{ 1 , \ldots , d\}$, the operator $\B^{\infty} ( j_{0} , k_{0} , n ; B ) = ( \B^{\infty}_{j , k} )_{j , k = 1 , \ldots , d}$ is defined by
\begin{gather}
\B^{\infty}_{j , k} =0 \qquad \text{if }( j , k ) \notin \{ ( n , n ) , ( j_{0} , k_{0} ) , ( k_{0} , j_{0} ) \}   \nonumber \\
\B^{\infty}_{j_{0} , k_{0}} = - ( d_{\phi , h}^{n} )^{*} B d_{\phi , h}^{n} \qquad \text{and} \qquad \B^{\infty}_{k_{0} , j_{0}} = ( \B^{\infty}_{j_{0} , k_{0}} )^{*}   \label{e27} \\
\B^{\infty}_{n , n} = ( d_{\phi , h}^{j_{0}} )^{*}B d_{\phi , h}^{k_{0}} + ( d_{\phi , h}^{k_{0}} )^{*} B^{*} d_{\phi , h}^{j_{0}} . \nonumber
\end{gather}
When $j_{0} = k_{0}$, we use the convention that $\B^{\infty}_{j_{0} , j_{0}} = - ( d_{\phi , h}^{n} )^{*} ( B + B^{*} ) d_{\phi , h}^{n}$. Such modifications will be used away from the critical points.

\item[$\bullet$] Exchange between four coefficients. For any $j_{0} , k_{0} , k_{1} \in \{ 1 , \ldots , d \}$, the operator $\B^{0} ( j_{0} , k_{0} , k_{1} ; B ) = ( \B^{0}_{j , k} )_{j , k = 1 , \ldots , d}$ is defined by
\begin{gather}
\B^{0}_{j , k} = 0 \qquad \text{if } ( j , k ) \notin \{ ( j_{0} , k_{0} ) , ( k_{0} , j_{0} ) , ( j_{0} , k_{1} ) , ( k_{1} , j_{0} ) \}  \nonumber \\
\B_{j_{0} , k_{0}}^{0} = - B d_{\phi , h}^{k_{1}} \qquad \text{and} \qquad \B_{k_{0} , j_{0}}^{0} = ( \B_{j_{0} , k_{0}}^{0} )^{*}  \label{e28} \\
\B_{j_{0} , k_{1}}^{0} = B d_{\phi , h}^{k_{0}} \qquad \text{and} \qquad \B_{k_{1} , j_{0}}^{0} = ( \B_{j_{0} , k_{1}}^{0} )^{*} . \nonumber
\end{gather}
Such modifications will be used near the critical points.
\end{itemize}
Recall that the $d$-matrix-weights $\A$ and $\D \A$ are given by $\A_{j , k} = \< \xi_{j} \>^{- 1} \< \xi_{k} \>^{- 1}$ and $( \D \A )_{j , k} = \< \xi_{k} \>^{- 1}$. Using the preceding remark, we can prove the following

\begin{lemma}\sl
Let $\widehat{Q} = \Op ( \hat{q} )$ where $\hat{q} \in S^{0} ( \A )$ is a Hermitian symbol such that $\hat q ( x , \xi ; h) = \hat q_{0} ( x , \xi ) + S^{0} ( h \A )$. We denote $P = d_{\phi , h}^{*} \widehat{Q} d_{\phi , h}$ and $p ( x , \xi ; h) = p_{0} ( x , \xi ) + S^{0} (h) \in S^{0} (1)$ its symbol. Assume that the following assumptions hold:
\begin{itemize}
\item[(A1)] $\forall \delta > 0 , \ \exists \alpha > 0 , \ \forall ( x , \xi ) \in T^{*} \R^{d} , \quad ( \vert \xi \vert^{2} + d ( x , \uuu )^{2} \geq \delta \Longrightarrow p_{0} ( x , \xi ) \geq \alpha )$.

\item[(A2)] Near $( \u , 0 )$ for any critical point $\u \in \uuu$, we have
\begin{equation} \label{e15}
p_{0} ( x , \xi ) = \vert \xi \vert^{2} + \vert \nabla \phi (x) \vert^{2} + r ( x , \xi ) ,
\end{equation}
with $r ( x , \xi ) = \ooo ( \vert ( x - \u , \xi ) \vert^{3} )$.
\end{itemize}
Then, for $h$ small enough, there exists a symbol $q \in S^{0} ( \D \A )$ such that
\begin{equation*}
P_{h} = d_{\phi , h}^{*} Q^{*} Q d_{\phi , h} ,
\end{equation*}
with $Q = \Op ( q )$ and
\begin{equation}\label{e115}
q ( x , \xi ; h ) = \Id + \ooo ( \vert ( x - \u , \xi ) \vert ) + S^{0} (h) ,
\end{equation}
near $( \u , 0 )$ for any $\u \in \uuu$. Moreover,  $Q = F \Op ( \Xi^{- 1} )$ for some $F \in \Psi^{0} (1)$ invertible and self-adjoint with $F^{- 1} \in \Psi^{0} (1)$.

If, additionally to the previous assumptions, we suppose
\begin{itemize}
\item[(A2')] the remainder term in \eqref{e15} satisfies $r ( x , \xi ) = \ooo ( \vert ( x - \u , \xi ) \vert^{4} )$ ,
\end{itemize}
then
\be\label{e1115}
q ( x , \xi ; h ) = \Id + \ooo ( \vert ( x - \u , \xi) \vert^{2} ) + S^{0} (h)
\ee
near $(\u,0)$.

Eventually, if $\widehat{q} \in S^{0}_{cl} ( \A )$ then $q \in S^{0}_{cl} ( \D \aaa )$.
\end{lemma}

\begin{proof}
In the following, we assume that $\phi$ has a unique critical point $\u$ and that $\u = 0$. Using some cutoff in space, we can always make this assumption without loss of generality. Given $\varepsilon > 0$, let $w_{0} , w_{1} , \ldots , w_{d} \in S^{0} (1)$ be non-negative functions such that
\begin{equation}
w_{0} + w_{1} + \cdots + w_{d} = 1 ,
\end{equation}
whose support satisfies
\begin{equation*}
\supp ( w_{0} ) \subset \big\{ \vert \xi \vert^{2} + \vert \nabla \phi (x) \vert^{2} \leq 2 \varepsilon \big\} ,
\end{equation*}
and, for all $\ell \geq 1$,
\begin{equation*}
\supp ( w_{\ell} ) \subset \Big\{ \vert \xi \vert^{2} + \vert \nabla \phi (x) \vert^{2} \geq \varepsilon \text{ and } \vert \xi_{\ell} \vert^{2} + \vert \partial_{\ell} \phi (x) \vert^{2} \geq \frac{1}{2 d} \big( \vert \xi \vert^{2} + \vert \nabla \phi (x) \vert^{2} \big) \Big\} .
\end{equation*}
Let us decompose $\widehat{Q}$ according to these truncations
\begin{equation} \label{e34}
\widehat{Q} = \sum_{\ell = 0}^{d} \widehat{Q}^{\ell} ,
\end{equation}
with $\widehat{Q}^{\ell} : = \Op ( w_{\ell} \widehat{q} )$ for all $\ell \geq 0$. We will modify each of the operators $\widehat{Q}^{\ell}$ separately, using the following modifiers. For $j_{0} , k_{0} , n \in \{1 , \ldots , d \}$ and $\beta \in S^{0} ( \< \xi_{j_{0}} \>^{- 1} \< \xi_{k_{0}} \>^{- 1} \< \xi_{n} \>^{- 2 })$ we denote for short
\begin{equation*}
\B^{\infty} ( j_{0} , k_{0} , n ; \beta ) : = \B^{\infty} ( j_{0} , k_{0} , n ; \Op ( \beta ) ) ,
\end{equation*}
where the right hand side is defined by \eqref{e27}. The same way, given $j_{0} , k_{0} , k_{1} \in \{ 1 , \ldots , d \}$ and $\beta \in S^{0} ( \< \xi_{j_{0}} \>^{- 1} \< \xi_{k_{0}} \>^{- 1} \< \xi_{k_{1}} \>^{- 1} )$ we denote for short
\begin{equation*}
\B^{0} ( j_{0} , k_{0} , k_{1} ; \beta ) : = \B^{0} ( j_{0} , k_{0} , k_{1} ; \Op ( \beta ) ) ,
\end{equation*}
where the right hand side is defined by \eqref{e28}. Observe that any operator of one of these two forms belongs to $\Psi^{0} ( \A )$. Let $\M ( \A ) \subset \Psi^{0} ( \A )$ be the vector space of bounded operators on ${L^{2} ( \R^{d} )}^{d}$ generated by such operators. Then, \eqref{e14} says exactly that
\begin{equation} \label{e14}
P_{h} = d_{\phi , h}^{*} ( \widehat{Q} + \mmm ) d_{\phi , h} ,
\end{equation}
for any $\mmm \in \M ( \A )$.

\medskip
{\bf Step 1.} We first remove the terms of order $1$ near the origin. More precisely, we show that there exists $\mmm^{0} \in \M ( \A )$ such that
\begin{equation} \label{e31}
\breve{Q}^{0} : = \widehat{Q}^{0}+\mmm^{0} = \Op ( \breve{q}^{0}) + \Psi^{0} ( h \A ) ,
\end{equation}
where $\breve{q}^{0} \in S^{0} ( \A )$ satisfies, near $( 0 , 0 ) \in T^{*} \R^{d}$,
\begin{equation} \label{e32}
\breve{q}^{0} ( x , \xi ) = w_{0} ( x , \xi ) \big( \Id + \rho ( x , \xi ) \big) ,
\end{equation}
with $\rho \in S ( \A )$ such that
\begin{itemize}
\item[$\bullet$] $\rho ( x , \xi ) = \ooo ( \vert ( x , \xi ) \vert )$ under the assumption (A2).
\item[$\bullet$] $\rho ( x , \xi ) = \ooo ( \vert ( x , \xi ) \vert^2 )$ under the assumption (A2').
\end{itemize}

From \eqref{e13}, we have
\begin{equation*}
p_{0} ( x , \xi ) = \sum_{j , k = 1}^{d} \widehat{q}_{0 ; j , k} ( x , \xi ) \big( \xi_{j} + i \partial_{j} \phi (x) \big) \big( \xi_{k} - i \partial_{k} \phi (x) \big) ,
\end{equation*}
where $\widehat{q}_{0} = ( \widehat{q}_{0 ; j , k})_{j , k}$ denotes the principal symbol of $\widehat{q}$. Expanding $\widehat{q}_{0}$ near the origin, we get
\begin{equation*}
\widehat{q}_{0} ( x , \xi ) = \widehat{q}_{0} ( 0 , 0 ) + \nu ( x , \xi ) ,
\end{equation*}
with $\nu ( x , \xi ) = \ooo ( \vert ( x , \xi ) \vert )$. Then, we deduce
\begin{equation} \label{e16}
p_{0} ( x , \xi ) = \sum_{j , k = 1}^{d} \big( \widehat{q}_{0 ; j , k} ( 0 , 0 ) + \nu_{j , k} ( x , \xi ) \big) \big( \xi_{j} + i \partial_{j} \phi (x) \big) \big( \xi_{k} - i \partial_{k} \phi (x) \big).
\end{equation}

Identifying \eqref{e15} and \eqref{e16}, we obtain $\widehat{q}_{0 ; j , k} ( 0 , 0 ) = \delta_{j , k}$, which establishes \eqref{e31}--\eqref{e32} under the assumption (A2).

Suppose now that (A2') is satisfied. Identifying \eqref{e15} and \eqref{e16} as before, we obtain
\begin{equation} \label{e17}
\sum_{j , k = 1}^{d} \nu_{j , k} ( x , \xi ) \big( \xi_{j} + i \partial_{j} \phi (x) \big) \big( \xi_{k} - i \partial_{k} \phi (x) \big) = \ooo ( \vert ( x , \xi ) \vert^{4} ) .
\end{equation}
Denoting $A : = \hess ( \phi ) (0)$, we have $\partial_{j} \phi ( x ) = ( A x )_{j} + \ooo ( x^{2} )$. Then, \eqref{e17} becomes
\begin{equation} \label{e18}
\sum_{j , k = 1}^{d} \nu_{j , k} ( x , \xi ) \big( \xi_{j} + i ( A x )_{j} \big) \big( \xi_{k} - i ( A x )_{k} \big) = \ooo ( \vert ( x , \xi ) \vert^{4} ) .
\end{equation}
Let us introduce the new variables $\eta = \xi + i A x$ and $\overline{\eta} = \xi - i A x$. Thus, \eqref{e18} reads
\begin{equation} \label{e19}
\sum_{j , k = 1}^{d} \nu_{j , k} ( x , \xi ) \eta_{j} \overline{\eta}_{k} = \ooo ( \vert ( x , \xi ) \vert^{4} ) = \ooo ( \vert ( \eta , \overline{\eta} ) \vert^{4} ) .
\end{equation}
On the other hand, since $A$ is invertible, there exist some complex numbers $\alpha_{j , k}^{n} , \widetilde{\alpha}_{j , k}^{n}$, for $j , k , n = 1 , \ldots , d$, such that
\begin{equation}\label{e30}
\nu_{j , k} ( x , \xi ) = \sum_{n = 1}^{d} ( \alpha_{j , k}^{n} \overline{\eta}_{n} + \widetilde{\alpha}_{j , k}^{n} \eta_{n} ) + \ooo ( \vert ( \eta , \overline{\eta} ) \vert^{2} ) .
\end{equation}
Combined with \eqref{e19}, this  yields $\sum_{j , k , n = 1}^{d} ( \alpha_{j , k}^{n} \overline{\eta}_{n} + \widetilde{\alpha}_{j , k}^{n} \eta_{n} ) \eta_{j} \overline{\eta}_{k} = \ooo ( \vert ( \eta , \overline{\eta} ) \vert^{4} )$ and since the left hand side is a polynomial of degree $3$ in $( \eta , \overline{\eta} )$, it follows that
\begin{equation}
\sum_{j , k , n = 1}^{d} \big( \alpha_{j , k}^{n} \overline{\eta}_{n} + \widetilde{\alpha}_{j , k}^{n} \eta_{n} \big) \eta_{j} \overline{\eta}_{k} = 0 ,
\end{equation}
for any $\eta \in \C^{d}$. Hence, uniqueness of coefficients of polynomials of $( \eta , \overline{\eta} )$ implies
\begin{equation} \label{e20}
\forall j , k  , n \in \{ 1 , \ldots , d \} , \qquad \alpha_{j , k}^{n} + \alpha_{j , n}^{k} = 0 .
\end{equation}
In particular, $\alpha_{j , k}^{k} = 0$. On the other hand, $\widetilde{\alpha}_{j , k}^{n} = \overline{\alpha_{k , j}^{n}}$ for all $k , j , n$ since $\widehat{Q}$ is selfadjoint.

Now, we define
\begin{equation*}
\breve{Q}^{0} : = \widehat{Q}^{0} + \mmm^{0} \qquad \text{with} \qquad \mmm^{0} : = \sum_{j_{0} , k_{0} = 1}^{d} \sum_{n = k_{0} + 1}^{d} \alpha_{j_{0} , k_{0}}^{n} \B^{0} ( j_{0} , k_{0} , n ; w_{0} ) .
\end{equation*}
It follows from symbolic calculus that $\breve{Q}^{0} = \Op ( \breve{q}^{0} )$ with $\breve{q}^{0} \in S^{0} ( \A )$ given by
\begin{align*}
\breve{q}^{0}_{j , k} = w_{0} \Big( \widehat{q}_{0 ; j , k} - \sum_{n > k} & \alpha_{j , k}^{n} \big( \xi_{n} - i \partial_{n} \phi (x) \big) + \sum_{n < k} \alpha_{j , n}^{k} \big( \xi_{n} - i \partial_{n} \phi (x) \big)   \\
&- \sum_{n > j} \overline{\alpha_{k , j}^{n}} \big( \xi_{n} + i \partial_{n} \phi (x) \big) + \sum_{n < j} \overline{\alpha_{k , n}^{j}} \big( \xi_{n} + i \partial_{n} \phi (x) \big) \Big) + S^{0} (h\A) ,
\end{align*}
for any $j , k$. Moreover, from \eqref{e20} and $\xi_{n} + i \partial_{n} \phi (x) = \eta_{n} + \ooo ( \vert x \vert^{2} )$ near $( 0 , 0 )$, we get
\begin{equation*}
\breve{q}^{0}_{j , k} = w_{0} \Big( \widehat{q}_{0 ; j , k} - \sum_{n = 1}^{d} \alpha_{j , k}^{n} \overline{\eta}_{n} - \sum_{n = 1}^{d} \overline{\alpha_{k , j}^{n}} \eta_{n}  + \widetilde{\rho}_{j , k}\Big) + S^{0} ( h \A ) ,
\end{equation*}
with $\widetilde{\rho} \in S^{0} ( \A )$ such that $\widetilde{\rho} = \ooo ( \vert ( x , \xi ) \vert^{2} )$ near the origin. Using the identity $\widehat{q}_{0 ; j , k} = \delta_{j , k} + \nu_{j , k}$ together with \eqref{e30}, we get
\begin{equation*}
\breve{q}^{0}_{j , k} = w_{0} \big( \delta_{j , k} + \rho_{j , k} \big) + S^{0} ( h \A ) ,
\end{equation*}
with $\rho \in S^{0} ( \A )$ such that $\rho = \ooo ( \vert ( x , \xi ) \vert^{2} )$ near the origin. This implies \eqref{e31}--\eqref{e32} under the assumption (A2') and achieves the proof of Step 1.

\medskip
{\bf Step 2.} We now remove the antidiagonal terms away from the origin. More precisely, we show that there exist some $\mmm^{\ell} \in \M ( \A )$ and some diagonal symbols $\widetilde{q}^{\ell} \in S^{0} ( \A )$ such that
\begin{equation} \label{e21}
\widetilde{Q}^{\ell} : = \widehat{Q}^{\ell} + \mmm^{\ell} = \Op ( w_{\ell} \widetilde{q}^{\ell} ) + \Psi^{0} ( h \A ) ,
\end{equation}
for any $\ell \in \{ 1 , \ldots , d \}$.

For $j_{0} , k_{0} , \ell \in \{ 1 , \ldots , d \}$ with $j_{0} \neq k_{0}$, let $\beta_{j_{0} , k_{0},\ell}$ be defined by
\begin{equation*}
\beta_{j_{0} , k_{0} , \ell} ( x , \xi ) : = \frac{w_{\ell} ( x , \xi ) \widehat{q}_{j_{0} , k_{0}} ( x , \xi )}{\vert \xi_{\ell} \vert^{2} + \vert \partial_{\ell} \phi(x) \vert^{2}} .
\end{equation*}
Thanks to the support properties of $w_{\ell}$, we have $\beta_{j_{0} , k_{0} , \ell} \in S^{0} ( \< \xi_{j_{0}} \>^{- 1} \< \xi_{k_{0}} \>^{- 1} \< \xi_{\ell} \>^{- 2} )$ so that $\B^{\infty} ( j_{0} , k_{0} , \ell ; \beta_{j_{0} , k_{0} , \ell} )$ belongs to $\M ( \A )$. Defining
\begin{equation*}
\mmm^{\ell} : = \sum_{j_{0} \neq k_{0}} \B^{\infty} ( j_{0} , k_{0} , \ell ; \beta_{j_{0} , k_{0} , \ell} ) ,
\end{equation*}
the pseudodifferential calculus gives
\begin{equation*}
( d_{\phi , h}^{\ell} )^{*} \Op ( \beta_{j_{0} , k_{0} , \ell} ) d_{\phi , h}^{\ell} =  \Op ( w_{\ell} \widehat{q}_{j_{0} , k_{0}} ) + \Psi^{0} \big( h \< \xi_{j_{0}} \>^{- 1} \< \xi_{k_{0}} \>^{- 1} \big) ,
\end{equation*}
which implies
\begin{equation*}
\widehat{Q}^{\ell} + \mmm^{\ell} = \Op ( w_{\ell} \widehat{q} ) + \mmm^{\ell} = \Op ( w_{\ell} \widetilde{q}^{\ell} ) + \Psi^{0} (h \A ) ,
\end{equation*}
with $\widetilde{q}^{\ell} \in S^{0} ( \A )$ diagonal. This proves \eqref{e21}.

\medskip
{\bf Step 3.} Let us now prove that we can modify each $\widetilde{Q}^{\ell}$ in order that its diagonal coefficients are suitably bounded from below. More precisely, we claim that there exist $c > 0$ and $\widetilde{\mmm}^{\ell} \in \M ( \A )$ such that
\begin{equation} \label{e22}
\breve{Q}^{\ell} : = \widetilde{Q}^{\ell} + \widetilde{\mmm}^{\ell} = \Op ( \breve{q}^{\ell} ) + \Psi^{0} ( h \A ) ,
\end{equation}
with $\breve{q}^{\ell}$ diagonal and $\breve{q}^{\ell}_{i_{0} , i_{0}} ( x , \xi ) \geq c w_{\ell} ( x , \xi ) \< \xi_{i_{0}} \>^{- 2}$ for all $i_{0} \in \{ 1 , \ldots , d \}$.

For $\ell , i_{0} \in \{ 1 , \ldots , d \}$, let $\beta_{i_{0} , \ell}$ be defined by
\begin{equation*}
\beta_{i_{0} , \ell} ( x , \xi ) : = \frac{w_{\ell} ( x , \xi )}{2 \big( \vert \xi_{\ell} \vert^{2} + \vert \partial_{\ell} \phi (x) \vert^{2} \big)} \bigg( \widetilde{q}^{\ell}_{i_{0} , i_{0}} ( x , \xi ) - \frac{\gamma}{1 + \vert \xi_{i_{0}} \vert^{2} + \vert \partial_{i_{0}} \phi ( x ) \vert^{2}} \bigg) ,
\end{equation*}
where $\gamma > 0$ will be specified later. The symbol $\beta_{i_{0} , \ell}$ belongs to $S^{0} ( \< \xi_{i_{0}} \>^{- 2} \< \xi_{\ell} \>^{- 2} )$ so that $\B^{\infty} ( i_{0} , i_{0} , \ell ; \beta_{i_{0} , \ell} ) \in \M ( \A )$. Defining
\begin{equation*}
\widetilde{\mmm}^{\ell} : = \sum_{i_{0} \neq \ell} \B^{\infty} ( i_{0} , i_{0} , \ell ; \beta_{i_{0} , \ell} ) ,
\end{equation*}
the symbolic calculus shows that $\widetilde{Q}^{\ell} + \widetilde{\mmm}^{\ell} = \Op ( \breve{q}^{\ell} ) + \Psi^{0} ( h \A )$ with $\breve{q}^{\ell}$ diagonal and
\begin{equation} \label{e23}
\forall i_{0} \neq \ell , \qquad \breve{q}^{\ell}_{i_{0} , i_{0}} ( x , \xi ) = \frac{\gamma w_{\ell} ( x , \xi )}{1 + \vert \xi_{i_{0}} \vert^{2} + \vert \partial_{i_{0}} \phi (x) \vert^{2}} .
\end{equation}

It remains to prove that we can choose $\gamma > 0$ above so that $\breve{q}^{\ell}_{\ell , \ell} ( x , \xi ) \geq c w_{\ell} ( x , \xi ) \< \xi_{\ell} \>^{- 2}$. Thanks to assumption (A1), there exists $\alpha > 0$ such that
\begin{equation} \label{e24}
\forall ( x , \xi ) \in \supp( w_{\ell} ) , \qquad p_{0} ( x , \xi ) \geq \alpha .
\end{equation}
On the other hand, a simple commutator computation shows that $\Op ( w_{\ell} ) P_{h} = d_{\phi , h}^{*} \widehat{Q}^{\ell} d_{\phi , h} + \Psi^{0} ( h )$. Combined with \eqref{e14}, \eqref{e21} and the definition of $\breve{q}^{\ell}$, this yields to
\begin{equation*}
\Op ( w_{\ell} ) P_{h} = d_{\phi , h}^{*} \widetilde{Q}^{\ell} d_{\phi , h} + \Psi^{0} ( h ) = d_{\phi , h}^{*} \Op ( \breve{q}^{\ell} ) d_{\phi , h} + \Psi^{0} ( h ) ,
\end{equation*}
and then
\begin{equation*}
( w_{\ell} p_{0} ) ( x , \xi ) = \sum_{i_{0} = 1}^{d} \breve{q}^{\ell}_{i_{0} , i_{0}} ( x , \xi ) \big( \vert \xi_{i_{0}} \vert^{2} + \vert \partial_{i_{0}} \phi (x) \vert^{2} \big) + S^{0} (h) .
\end{equation*}
Using now \eqref{e23}, we get
\begin{equation*}
( w_{\ell} p_{0} ) ( x , \xi ) = \breve{q}^{\ell}_{\ell , \ell} ( x , \xi ) \big( \vert \xi_{\ell} \vert^{2} + \vert \partial_{\ell} \phi(x) \vert^{2} \big) + \gamma ( d - 1 ) w_{\ell} ( x , \xi ) + S^{0} (h) .
\end{equation*}
Combining this relation with \eqref{e24} and choosing $\gamma = \frac{\alpha}{2 d}$, we obtain
\begin{equation} \label{e29}
\breve{q}^{\ell}_{\ell , \ell} ( x , \xi ) \geq \frac{\alpha w_{\ell} ( x,\xi )}{2 \big( \vert \xi_{\ell} \vert^{2} + \vert \partial_{\ell} \phi ( x ) \vert^{2} \big)} + S^{0} \big( h \< \xi_{\ell} \>^{- 2} \big) .
\end{equation}
Thus, $\breve{q}^{\ell}_{\ell , \ell}$ satisfies the required lower bound and \eqref{e22} follows.

\medskip
{\bf Step 4.} Eventually, we take the square root of the modified operator. Let us define
\begin{equation}
\breve{Q} : = \sum_{\ell = 0}^{d} \breve{Q}^{\ell} \in \Psi^{0} ( \A ) ,
\end{equation}
with $\breve{Q}^{\ell}$ defined above. Thanks to \eqref{e14}, we have  $P_{h} = ( d_{\phi , h} )^{*} \widehat{Q} d_{\phi , h} = ( d_{\phi , h} )^{*} \breve{Q} d_{\phi , h}$ and it follows from the preceding constructions that the principal symbol $\breve{q}$ of $\breve{Q}$ satisfies
\begin{equation*}
\breve{q} ( x , \xi ) \geq w_{0} ( x , \xi ) \big( \Id + \ooo ( \vert ( x , \xi ) \vert ) \big) + c \sum_{\ell \geq 1} w_{\ell} ( x , \xi ) \diag \big( \<\xi_{j} \>^{- 2} \big) .
\end{equation*}
Shrinking $c > 0$ and the support of $w_{0}$ if necessary, it follows that
\begin{equation*}
\breve{q} ( x , \xi ) \geq c \diag \big( \<\xi_{j} \>^{-2} \big) .
\end{equation*}
Denoting $E = \Op ( \D ) \breve{Q} \Op ( \D )$ and $e \in S^{0} (1)$ the symbol of $E$, the pseudodifferential calculus gives $e(x,\xi;h)=e_0(x,\xi)+S^0(h)$ with
\begin{equation} \label{e25}
e_0 ( x , \xi ) \geq c \diag \big( \<\xi_{j} \> \<\xi_{j} \>^{- 2} \<\xi_{j} \> \big)=c\Id ,
\end{equation}
so that
for $h>0$ small enough, $e (x,\xi)\geq \frac c 2 \Id$. Hence, we can adapt the proof of Theorem 4.8 of \cite{HeNi05_01} to our semiclassical setting 
to get that $F : = E^{1 / 2}$ belongs to $\Psi^{0} (1)$ and that
$F^{- 1} \in \Psi^{0} (1)$. 
Then, $\breve{Q} = Q^{*} Q$ with $Q : = F \Op ( \Xi^{- 1} )$ and by construction $Q \in \Psi^{0} ( \Xi \A)$. 

In addition, as in Theorem 4.8 of \cite{HeNi05_01}, we can show that  $F=\Op(e_0^{1/2})+\Psi^0(h)$ so that
$Q=\Op(q_0)+\Psi^0(h\Xi\A)$ with $q_0=e_0^{1/2}\Xi^{- 1}$.
If moreover $\hat q $ admits a classical expansion, then
$\breve q\in S^0_{cl}(\A)$ and the same argument shows that both $e$ and $q$ admit classical expansion.

Let us now study $q_0$ near $(\u,0)$.
Since for $(x,\xi)$ close to   $(\u,0)$ we have $\Xi=\Id+\ooo(\vert\xi\vert^2)$ and $\breve q_0=\Id+\rho(x,\xi)$,   then
$$e_0(x,\xi)=\Xi\breve q_0\Xi=\Id+\rho(x,\xi)+\ooo(|\xi|^2)$$
and we get easily 
$q_0=e_0^{1/2}\Xi^{- 1}=\Id+\ooo(|\xi|^2+\rho(x,\xi))$ which proves \eqref{e115}  and \eqref{e1115}.
\end{proof}

\section{Quasimodes on $k$-forms and first exponential type eigenvalue estimates} \label{s4}

\Subsection{Pseudodifferential Hodge--Witten Laplacian on the $0$-forms}

This part is devoted to the rough asymptotic of the small eigenvalues of $P_h$ and to the construction of associated quasimodes. From Theorem \ref{e4}, this operator has the following expression
\begin{equation} \label{b34}
P_h = a_{h} d_{\phi , h}^{*} G d_{\phi , h} a_{h} ,
\end{equation}
where $G$ is the matrix of pseudodifferential operators
\begin{equation*}
G = \big( \Op ( g_{j , k} ) \big)_{j , k} : = Q^{*} Q = \Op(\Xi)^{-1} E^{*} E \Op(\Xi)^{-1} .
\end{equation*}
Using Corollary \ref{e8} and that $G$ is selfadjoint, we remark that $g_{j , k} \in S^{0} ( \< \xi_{j} \>^{- 1} \< \xi_{k} \>^{- 1} )$ and $g_{j , k} = \overline{g_{k , j}}$. Thus, $P_h$ can be viewed as a Hodge--Witten Laplacian on $0$-forms (or a Laplace--Beltrami operator) with the pseudodifferential metric $G^{-1}$. In the following, we will then use the notation $P^{(0)} : = P_h$.

Since $a_{h} ( \u ) = 1+\ooo(h)$ and $g ( \u , 0 )= \beta_{d} \Id + \ooo ( h )$ for all the critical points $\u \in \uuu$, it is natural to consider
the operator with the coefficients $a_{h}$ and $G$ frozen at $1$ and $\beta_d \Id$. For that, let $\Omega^{p} ( \R^{d} )$, $p = 1 , \ldots , d$,
be the space of the $C^{\infty}$ $p$-forms on $\R^{d}$. We then define
\begin{equation} \label{d27}
P^{W} = d_{\phi , h}^{*} d_{\phi , h} + d_{\phi , h} d_{\phi , h}^{*} ,
\end{equation}
the semiclassical Witten Laplacian on the De Rham complex, and $P^{W , (p)}$ its restriction to the $p$-forms.
This operator has been intensively studied (see e.g. \cite{HeSj85_01}, \cite{CyFrKiSi87_01}, \cite{BoEcGaKl04_01}, \cite{HeKlNi04_01}, \cite{BoGaKl05_01}, \ldots), and a lot is known concerning its spectral properties.
In particular, from Lemma 1.6 and Proposition 1.7 of Helffer-Sj\"{o}strand \cite{HeSj85_01}, we know that there are $n_{0}$ exponentially small (real non-negative) eigenvalues,
and that the other are above $h / C$. 

From \cite{HeKlNi04_01} and \cite{HeHiSj11_01}, we have good normalized quasimodes  for $P^{W , (0)}$ associated to each minima of $\phi$. For $k \in \{1, \ldots , n_0 \}$, they are given by
\begin{equation*}
f_k^{W,(0)}(x)=  \chi_{k , \varepsilon } (x) b^{(0)}_k(h) e^{- (\phi (x)-\phi (\m_k))/h},
\end{equation*}
where $ b^{(0)}_k(h) = (\pi h)^{-d/4} \det(\hess \phi(\m_k))^{1/4} (1 + \ooo(h) )$, 
and  where $\chi_{k , \varepsilon}$ are cutoff functions localized in sufficiently large areas containing  $\m_k \in \uuu^{(0)}$.
In fact we need   
large support (associated to level sets of $\phi$) and properties for the cutoff functions $\chi_{k , \varepsilon}$, so that the refined analysis of the next section can be done. We postpone to Appendix \ref{d1} the construction of the cutoff functions, the definition of $\varepsilon > 0$, refined estimates on this family $( f_k^{W,(0)} )_{k}$, and in particular the fact that it is a quasi-orthonormal free family of functions,  following closely \cite{HeKlNi04_01} and \cite{HeHiSj11_01}.

We now define the quasimodes associated to $P^{(0)}$ in the following way:
\begin{equation} \label{b35}
f_k^{(0)}(x) : = a_h(x)^{-1} f_k^{W,(0)}(x) = a_h(x)^{-1} b^{(0)}_k(h) \chi_{k , \varepsilon} (x) e^{- (\phi (x)-\phi (\m_k)) / h} ,
\end{equation}
for $1 \leq k \leq n_{0}$. We have then

\begin{lemma}\sl \label{d2}
The system $( f_k^{(0)} )_{k}$ is free and there exist $\alpha > 0$ independent of $\varepsilon$ such that
\begin{equation*}
\big\< f_k^{(0)} , f_{k'}^{(0)} \big\> = \delta_{k , k'} + \ooo (h) \qquad \text{and } \qquad P^{(0)} f_k^{(0)}  =  \ooo (e^{-\alpha/h}).
\end{equation*}
\end{lemma}

\begin{remark}\sl Note that, for this result to be true, it would have been sufficient to take truncation functions with smaller support (say in a small neighborhood of each minimum $\m_k$). We emphasize again that the more complicated construction for the quasimodes is justified by later purpose.
\end{remark}

\begin{proof}
First, observe that 
$$\<f_k^{(0)},f_{k'}^{(0)}\>=\<a_h^{-2}f_k^{W,(0)},f_{k'}^{W,(0)}\>=\delta_{kk'}+\<(a_h^{-2}-1)f_k^{W,(0)},f_{k'}^{W,(0)}\>.$$
Moreover, since near any minimum $\m_k$, $a_h^{-2}-1=\ooo(h+\vert x-\m_k\vert^2)$ and $\phi(x)-\phi(\m_k)$ is quadratic, then
\begin{equation} \label{z9}
\big\Vert (a_h^{-2}-1)f_k^{W,(0)} \big\Vert=\ooo(h) ,
\end{equation}
which proves the first statement.
For the last statement, this is enough to notice that
\begin{equation*}
P^{(0)} f_k^{(0)} = L_\phi^* Q d_{\phi , h} f_k^{W,(0)} 
\end{equation*}
and apply Lemma \ref{e7}.
\end{proof}

We prove now a first rough spectral result on $P^{(0)}$, using the preceding lemma.

\begin{proposition}\sl \label{b32}
The operator $P^{(0)}$ has exactly $n_{0}$ exponentially small (real non-negative) eigenvalues, and the remaining part of its spectrum is in $[ \varepsilon_0 h, + \infty [$, for some $\varepsilon_0 > 0$.
\end{proposition}

Usually, this type of result is a consequence of an IMS formula.
It is possible to do that here (with efforts), but we prefer to give a simpler proof using what we know about $P^{W , (0)}$.
The following proof is based on the spectral theorem and the maxi-min principle.

\begin{proof}
Thanks to Proposition \ref{a47}, the spectrum of $P^{(0)}$ is discrete in $[0 , \delta ]$ and its $j$-th eigenvalue is given by
\begin{equation} \label{b24}
\sup_{\dim E = j - 1} \quad \inf_{u \in E^{\perp} , \ \Vert u \Vert = 1} \quad \big\< P^{(0)} u , u \big\> .
\end{equation}
Lemma \ref{d2} directly implies
\begin{equation*}
\big\< P^{(0)} f_k^{(0)} , f_{k'}^{ (0)} \big\>  \leq \big\Vert P^{(0)} f_k^{(0)}\big\Vert \big\Vert f_{k'}^{ (0)} \big\Vert = \ooo ( e^{- \alpha / h} ) ,
\end{equation*}
for some $\alpha>0$. Using the almost orthogonality of the $f_{k}^{(0)}$'s, \eqref{b24} and $P^{(0)} \geq 0$, we deduce that $P^{(0)}$ has at least $n_{0}$ eigenvalues which are exponentially small.

We now want to prove that the remaining part of the spectrum of $P^{(0)}$ is above $\varepsilon_0 h$ for some $\varepsilon_0>0$ small enough. For this, we set
\begin{equation*}
\eee : = \vect \big\{ f_{k}^{W , (0)} ; \ k = 1 , \ldots , n_{0} \big\} ,
\end{equation*}
and we consider $u \in a_{h}^{- 1} \eee^{\perp}$ with $\Vert u \Vert = 1$. We have again,
\begin{equation} \label{b28}
\big\< P^{(0)} u , u \big\> = \big\Vert F \Op ( \Xi^{- 1} ) d_{\phi , h} a_{h} u \big\Vert^{2} \geq \varepsilon_0 \big\Vert \Op ( \Xi^{- 1} ) d_{\phi , h} a_{h} u \big\Vert^{2} ,
\end{equation}
for some $\varepsilon_0 > 0$ independent of $h$ which may change from line to line. For the last inequality, we have used that $\Vert F^{-1} \Vert$ is uniformly bounded since $F^{-1}\in \Psi^0(1)$. On the other hand, using $0 \leq P^{W, (1)} = - h^{2} \Delta \otimes \Id + \ooo ( 1 )$, we notice that
\begin{equation*}
\Op ( \Xi^{- 1} )^{2} \geq ( - h^{2} \Delta + 1 )^{- 1} \otimes \Id \geq \varepsilon_0 \big( P^{W , (1)} + 1 \big)^{- 1} ,
\end{equation*}
for some (other) $\varepsilon_0 > 0$. We therefore get, using the classical intertwining relations
\begin{equation*}
\big( P^{W , (1)} + 1 )^{- 1 / 2} d_{\phi , h} = d_{\phi , h} \big( P^{W , (0)} + 1 )^{- 1 / 2} ,
\end{equation*}
and the fact that $P^{W , (0)} = d_{\phi , h}^{*} d_{\phi , h}$ on $0$-forms,
\begin{align}
\big\< P^{(0)} u , u \big\> &\geq \varepsilon_0 \big\Vert \big( P^{W , (1)} + 1 \big)^{- 1 / 2} d_{\phi , h} a_{h} u \big\Vert^{2} = \varepsilon_0 \big\Vert d_{\phi , h} \big( P^{W , (0)} + 1 \big)^{- 1 / 2} a_{h} u \big\Vert^{2}  \nonumber \\
&= \varepsilon_0 \big\< P^{W , (0)} \big( P^{W , (0)} + 1 \big)^{- 1} a_{h} u , a_{h} u \big\> . \label{b31}
\end{align}
Now, let $\fff$ be the eigenspace of $P^{W , (0)}$ associated to the $n_{0}$ exponentially small eigenvalues,
and let $\Pi_{\eee}$ (resp. $\Pi_{\fff}$) be the orthogonal projectors onto $\eee$ (resp. $\fff$).
Then, from Proposition 1.7 of \cite{HeSj85_01} (see also Theorem 2.4 of \cite{HeSj84_01}), we have $\Vert \Pi_{\eee} - \Pi_{\fff} \Vert = \ooo ( e^{- \alpha/ h} )$.
Moreover, since the $n_{0} + 1$-th eigenvalue of $P^{W , (0)}$ is of order $h$, the spectral theorem gives
\begin{equation*}
P^{W , (0)} \big( P^{W , (0)} + 1 \big)^{- 1} \geq \varepsilon_0 h ( 1 - \Pi_{\fff} )+\ooo(e^{-\alpha/h}) \geq \varepsilon_0 h ( 1 - \Pi_{\eee} ) + \ooo ( e^{- \alpha / h} ) .
\end{equation*}
Then, using $a_{h} u \in \eee^{\perp}$, $\Vert u \Vert = 1$ and Lemma \ref{a5}, \eqref{b31} becomes
\begin{equation*}
\big\< P^{(0)} u , u \big\> \geq \varepsilon_0 h \Vert a_{h} u \Vert + \ooo ( e^{- \alpha/ h} ) \geq \frac{c_1 \varepsilon_0}{2} h .
\end{equation*}
Eventually, this estimate and \eqref{b24} imply that $P^{(0)}$ has at most $n_{0}$ eigenvalues below $\frac{c_1 \varepsilon_0}{2} h$. 
Taking  $c_1\varepsilon_0/2$ as a new value of $\varepsilon_0$ gives the result.
\end{proof}

\Subsection{Pseudodifferential Hodge--Witten Laplacian on the $1$-forms}

Since we want to follow a supersymmetric approach to prove the main theorem of this paper, we have to build an extension $P^{(1)}$ of $P^{(0)}$ defined on $1$-forms which satisfies properties similar to those of $P^{W , (1)}$. To do this, we use the following coordinates for $\omega \in \Omega^{1} ( \R^{d} )$ and $\sigma \in \Omega^{2} ( \R^{d} )$:
\begin{equation*}
\omega = \sum_{j = 1}^{d} \omega_{j} (x) d x_{j}, \qquad \sigma = \sum_{j < k} \sigma_{j , k} (x) d x_{j} \wedge d x_{k} ,
\end{equation*}
and we extend the matrix $\sigma_{j , k}$ as a function with values in the space of the antisymmetric matrices. Recall that the exterior derivative satisfies
\begin{equation} \label{b8}
\big( d^{(1)} \omega \big)_{j , k} = \partial_{x_{j}} \omega_{k} - \partial_{x_{k}} \omega_{j} \qquad \text{and} \qquad \big( d^{* (1)} \sigma \big)_{j} = - \sum_{k} \partial_{x_{k}} \sigma_{k , j} .
\end{equation}

In the previous section, we have seen that $P^{(0)}$ can be viewed as the Hodge--Witten Laplacian on $0$-forms with a pseudodifferential metric $G^{-1}$. It is then natural to consider the corresponding Hodge--Witten Laplacian on $1$-forms. Thus, mimicking the construction in the standard case, we define
\begin{equation} \label{b1}
P^{(1)} : = Q d_{\phi , h} a_h^2 d_{\phi , h}^{*} Q^{*} + ( Q^{- 1} )^{*} d_{\phi , h}^{*} M d_{\phi , h} Q^{-1} ,
\end{equation}
where $M$ is the linear operator acting on $\Omega^{2} ( \R^{d} )$ with coefficients
\begin{equation} \label{b5}
M_{( j , k ) , ( a , b )} : = \frac{1}{2} \Op \big( a_{h}^{2} \big( g_{j , a} g_{k ,b} - g_{k , a} g_{j , b} \big) \big) .
\end{equation}
Note that $M$ is well defined on $\Omega^{2} ( \R^{d} )$ (i.e. $M \sigma$ is antisymmetric if $\sigma$ is antisymmetric) since $M_{( k , j ) , ( a , b )} = M_{( j , k ) , ( b , a )} = - M_{( j , k ) , ( a , b )}$. Furthermore, we deduce from the properties of $g_{j,k}$ that
\begin{equation} \label{b6}
M_{( j , k ) , ( a , b )} \in \Psi^{0} \big( \< \xi_{j} \>^{- 1} \< \xi_{k} \>^{- 1} \< \xi_{a} \>^{- 1} \< \xi_{b} \>^{- 1} \big) .
\end{equation}

\begin{remark}\sl
When $G^{-1}$ is a true metric (and not a matrix of pseudodifferential operators), the operator $P^{(1)}$ defined in \eqref{b1} is the usual Hodge--Witten Laplacian on $1$-forms. Our construction is then an extension to the pseudodifferential case. Generalizing these structures to the $p$-forms, it should be possible to define a Hodge--Witten Laplacian on the total De Rham complex. It could also be possible to define such an operator using only abstract geometric quantities (and not explicit formulas like \eqref{b5}).

On the other hand, a precise choice for the operator $M$ is not relevant in the present paper. Indeed, for the study of the small eigenvalues of $P^{(0)}$, only the first part (in \eqref{b1}) of $P^{(1)}$ is important (see Lemma \ref{b23} below). The second part is only used to make the operator $P^{(1)}$ elliptic. Thus, any $M$ satisfying \eqref{b6} 
and $M_{( j , k ) , ( a ,b )} \geq \varepsilon \Op ( \< \xi_{j} \>^{- 2}  \< \xi_{k} \>^{-2} ) \otimes \Id$ should probably work.
\end{remark}

We first show that $P^{(1)}$ acts diagonally (at the first order) as it is the case 
for $P^{W , (1)}$.

\begin{lemma}\sl \label{b2}
The operator $P^{(1)} \in \Psi^{0} (1)$ is selfadjoint on $\Omega^{1} ( \R^{d} )$. Moreover,
\begin{equation} \label{b17}
P^{(1)} = P^{(0)} \otimes \Id + \Psi^{0} (h) .
\end{equation}
\end{lemma}

\begin{proof}
We begin by estimating the first part of $P^{(1)}$:
\begin{equation*}
P^{(1)}_{1} : = Q d_{\phi , h} a_h^2 d_{\phi , h}^{*} Q^{*} .
\end{equation*}
Let $q_{j , k} \in S^{0} ( \< \xi_{k} \>^{-1} )$ denote the symbol of the coefficients of $Q$ and let $d_{\phi , h}^{j} = h \partial_{j} + ( \partial_{j} \phi )$. Using the composition rules of matrices, a direct computation gives
\begin{align}
\big( P^{(1)}_{1} \big)_{j,k} &= \sum_{a} \Op ( q_{j , a} ) d_{\phi , h}^{a} a_{h}^{2} \big( d_{\phi , h}^{*} Q^{*} \big)_{k}   \nonumber \\
&= \sum_{a , b} \Op ( q_{j , a} ) d_{\phi , h}^{a} a_{h}^{2} ( d_{\phi , h}^{b} )^{*} \Op \big( \overline{q_{k , b}} \big) .  \label{d19}
\end{align}
We then deduce that $P^ {(1)}_{1}$ is a selfadjoint operator on $\Omega^{1} ( \R^{d} )$ with coefficients of class $\Psi^{0} (1)$. Moreover, this formula implies
\begin{equation} \label{b3}
\big( P^{(1)}_{1} \big)_{j,k} = \sum_{a , b} \Op \big( a_{h}^{2} q_{j , a} \overline{q_{k , b}} \big) d_{\phi , h}^{a} ( d_{\phi , h}^{b} )^{*} + \Psi^{0} (h) .
\end{equation}

It remains to study
\begin{equation*}
P^{(1)}_{2} : = ( Q^{- 1} )^{*} d_{\phi , h}^{*} M d_{\phi , h} Q^{-1} .
\end{equation*}
Let $q^{- 1}_{j , k} \in S^{0} ( \< \xi_{j} \> )$ denote the symbol of the coefficients of $Q^{- 1}$. The formulas of \eqref{b8}, the definition \eqref{b5} and the composition rules of matrices imply
\begin{align}
\big( P^{(1)}_{2} \big)_{j,k} &= \sum_{\alpha} \Op \big( \overline{q^{- 1}_{\alpha , j}} \big) \big( d_{\phi , h}^{*} M d_{\phi , h} Q^{-1} \big)_{\alpha , k}      \nonumber \\
&= - \sum_{a , \alpha} \Op \big( \overline{q^{- 1}_{\alpha , j}} \big) ( d_{\phi , h}^{a} )^{*} \big( M d_{\phi , h} Q^{-1} \big)_{(a , \alpha ) , k}      \nonumber \\
&= - \sum_{a , b , \alpha , \beta} \Op \big( \overline{q^{- 1}_{\alpha , j}} \big) ( d_{\phi , h}^{a} )^{*} M_{( a , \alpha ) , ( b , \beta )} \big( d_{\phi , h} Q^{-1} \big)_{( b , \beta ) , k}      \nonumber \\
&= - \sum_{a , b , \alpha , \beta} \Op \big( \overline{q^{- 1}_{\alpha , j}} \big) ( d_{\phi , h}^{a} )^{*} M_{( a , \alpha ) , ( b , \beta )} \big( d_{\phi , h}^{b} \Op ( q^{- 1}_{\beta , k} ) - d_{\phi , h}^{\beta} \Op ( q^{- 1}_{b , k} ) \big)    \nonumber \\
&= - 2 \sum_{a , b , \alpha , \beta} \Op \big( \overline{q^{- 1}_{\alpha , j}} \big) ( d_{\phi , h}^{a} )^{*} M_{( a , \alpha ) , ( b , \beta )} d_{\phi , h}^{b} \Op ( q^{- 1}_{\beta , k} ) ,  \label{b4}
\end{align}
where we have used that $M_{( a , \alpha ) , ( b , \beta )} = - M_{( a , \alpha ) , ( \beta , b )}$. From \eqref{b6}, a typical term of these sums satisfies
\begin{equation*}
\Op \big( \overline{q^{- 1}_{\alpha , j}} \big) ( d_{\phi , h}^{a} )^{*} M_{( a , \alpha ) , ( b , \beta )} d_{\phi , h}^{b} \Op ( q^{- 1}_{\beta , k} ) \in \Psi^{0} \big( \< \xi_{\alpha} \> \< \xi_{a} \> \< \xi_{a} \>^{- 1} \< \xi_{\alpha} \>^{- 1} \< \xi_{b} \>^{- 1} \< \xi_{\beta} \>^{- 1} \< \xi_{b} \> \< \xi_{\beta} \> \big) ,
\end{equation*}
and then $P^{(1)}_{2} \in \Psi^{0} (1)$. On the other hand, using $g_{j , k} = \overline{g_{k , j}}$ and \eqref{b5}, we get
\begin{align*}
\big( P^{(1)}_{2} \big)_{j,k}^{*} &= - \sum_{a , b , \alpha , \beta} \Op \big( \overline{q^{- 1}_{\beta , k}} \big) ( d_{\phi , h}^{b} )^{*} \Op \big( \overline{a_{h}^{2} ( g_{a , b} g_{\alpha , \beta} - g_{\alpha , b} g_{a , \beta}} ) \big) d_{\phi , h}^{a} \Op ( q^{- 1}_{\alpha , j} )   \\
&= - \sum_{a , b , \alpha , \beta} \Op \big( \overline{q^{- 1}_{\beta , k}} \big) ( d_{\phi , h}^{b} )^{*} \Op \big( a_{h}^{2} \big( g_{b , a} g_{\beta , \alpha} - g_{b , \alpha} g_{\beta , a} \big) \big) d_{\phi , h}^{a} \Op ( q^{- 1}_{\alpha , j} )   \\
&= - \sum_{a , b , \alpha , \beta} \Op \big( \overline{q^{- 1}_{\alpha , k}} \big) ( d_{\phi , h}^{a} )^{*} \Op \big( a_{h}^{2} \big( g_{a , b} g_{\alpha , \beta} - g_{a , \beta} g_{\alpha , b} \big) \big) d_{\phi , h}^{b} \Op ( q^{- 1}_{\beta , j} )  \\
&= \big( P^{(1)}_{2} \big)_{k , j} ,
\end{align*}
and then $P^{(1)}_{2}$ is selfadjoint on $\Omega^{1} ( \R^{d} )$. Finally, \eqref{b6} and \eqref{b4} yields
\begin{align}
\big( P^{(1)}_{2} \big)_{j,k} &= \sum_{a , b} \Op \Big( a_{h}^{2} \sum_{\alpha , \beta} \overline{q^{- 1}_{\alpha , j}} q^{- 1}_{\beta , k} \big( g_{a , b} g_{\alpha , \beta} - g_{a , \beta} g_{\alpha , b} \big) \Big) ( d_{\phi , h}^{a} )^{*} d_{\phi , h}^{b} + \Psi^{0} (h)    \nonumber \\
&= \sum_{a , b} \Op \big( a_{h}^{2} g_{a , b} \delta_{j , k} - a_{h}^{2} q_{j , b} \overline{q_{k , a}} \big) ( d_{\phi , h}^{a} )^{*} d_{\phi , h}^{b} + \Psi^{0} (h) , \label{b7}
\end{align}
since
\begin{equation*}
\sum_{j} g_{a , j} q^{- 1}_{j , b} = \overline{q_{b , a }} + S^{0} \big( h \< \xi_{a} \>^{- 1} \big) \qquad \text{and} \qquad \sum_{j} q_{a , j} q^{- 1}_{j , b} = \delta_{a , b} + S^{0} ( h ) ,
\end{equation*}
which follow from $G Q^{-1} = Q^{*}$ and $Q Q^{-1} = \Id$.

Summing up the previous properties of $P^{(1)}_{\bullet}$, the operator $P^{(1)} = P^{(1)}_{1} + P^{(1)}_{2} \in \Psi^{0} (1)$ is selfadjoint on $\Omega^{1} ( \R^{d} )$. Eventually, combining \eqref{b3} and \eqref{b7}, we obtain
\begin{align}
P^{(1)} &= \sum_{a , b} ( d_{\phi , h}^{a} )^{*} \Op \big( a_{h}^{2} g_{a , b} \big) d_{\phi , h}^{b} \otimes \Id + \Psi^{0} (h) = a_{h} d_{\phi , h}^{*} G d_{\phi , h} a_{h} \otimes \Id + \Psi^{0} (h)  \nonumber \\
&= P^{(0)}\otimes \Id + \Psi^{0} (h) ,
\end{align}
and the lemma follows.
\end{proof}

The next result compares $P^{(1)}$ and $P^{W , (1)}$.

\begin{lemma}\sl \label{d3}
There exist some pseudodifferential operators $( R_{k} )_{k = 0 , 1 , 2}$ such that
\begin{equation*}
P^{(1)} = \beta_{d} P^{W , (1)} + R_{0} + R_{1} + R_{2} ,
\end{equation*}
where the remainder terms enjoy the following properties:
\begin{enumerate}[i)]
\item $R_{0}$ is a $d \times d$ matrix whose coefficients are a finite sum of terms of the form
\begin{equation*}
( d_{\phi , h}^{a} )^{*} \big( \Op ( r_{0} ) + \Psi^{0} (h) \big) d_{\phi , h}^{b} ,
\end{equation*}
with $a , b \in \{ 1 , \ldots , d\}$ and $r_{0} \in S^{0} (1)$ satisfying $r_0 ( x , \xi ) = \ooo ( \vert ( x - \u , \xi ) \vert^2 )$ near $( \u , 0 )$, $\u \in \uuu$.
\item $R_{1}$ is a matrix whose coefficients are a finite sum of terms of the form $h \Op (r_1) d_{\phi , h}^{a}$ or $h ( d_{\phi , h}^{a} )^{*} \Op (r_1)$ with $a \in \{ 1 , \ldots , d\}$  and $r_1 \in S^0 (1)$ satisfying $r_1 ( x , \xi ) = \ooo ( \vert ( x - \u , \xi ) \vert )$ near $( \u , 0 )$, $\u \in \uuu$.
\item $R_2 \in \Psi^0 ( h^2 )$.
\end{enumerate}
\end{lemma}

\begin{proof}
As in the proof of Lemma \ref{b2}, we use the decomposition $P^{(1)} = P^{(1)}_{1} + P^{(1)}_{2}$. From Corollary \ref{e8} and Lemma \ref{a5}, the coefficients appearing in these operators verify
\begin{align*}
a_{h} &= \widetilde{a} + S^{0} ( h ) \in S^{0} ( 1 ) ,   \\
q_{a , b} &= \widetilde{q}_{a , b} + S^{0} ( h \< \xi_{b} \>^{- 1} ) \in S^{0} ( \< \xi_{b} \>^{- 1} ) ,   \\
q_{a , b}^{- 1} &= \widetilde{q}_{a , b}^{- 1} + S^{0} ( h \< \xi_{a} \> ) \in S^{0} ( \< \xi_{a} \> ) ,   \\
M_{( j , k ) , ( a , b )} &= \Op ( \widetilde{m}_{( j , k ) , ( a , b )} ) + \Psi^{0} \big( h \< \xi_{j} \>^{- 1} \< \xi_{k} \>^{- 1} \< \xi_{a} \>^{- 1} \< \xi_{b} \>^{- 1} \big) ,
\end{align*}
with $\widetilde{m}_{( j , k ) , ( a , b )} \in S^{0} ( \< \xi_{j} \>^{- 1} \< \xi_{k} \>^{- 1} \< \xi_{a} \>^{- 1} \< \xi_{b} \>^{- 1} )$ and
\begin{gather*}
\widetilde{a} = 1 + \ooo \big( \vert ( x - \u , \xi ) \vert^{2} \big) ,  \qquad  \widetilde{m}_{( j , k ) , ( a , b )} = \beta_{d}^{2} ( \delta_{j , a} \delta_{k ,b} - \delta_{k , a} \delta_{j , b} ) / 2 + \ooo \big( \vert ( x - \u , \xi ) \vert^{2} \big) ,  \\
\widetilde{q}_{a , b} = \beta_d^{1 / 2} \delta_{a , b} + \ooo \big( \vert ( x - \u , \xi ) \vert^{2} \big) , \qquad \widetilde{q}_{a , b}^{- 1} = \beta_d^{- 1 / 2} \delta_{a , b} + \ooo \big( \vert ( x - \u , \xi ) \vert^{2} \big) ,
\end{gather*}
near $( \u , 0 )$, $\u \in \uuu$. Then, making commutations in \eqref{d19} and \eqref{b4}, we obtain the announced result.
\end{proof}

We now make the link between the eigenvalues of $P^{(0)}$ and $P^{(1)}$. For that, we will use the so-called intertwining relations, which are one fundamental tool in the supersymmetric approach. Recall that, thanks to Theorem \ref{e4}, $P^{(0)}$ can be written as
\begin{equation} \label{d9}
P^{(0)} = L_{\phi}^{*} L_{\phi} \qquad \text{with} \qquad L_{\phi} = Q d_{\phi , h} a_{h} .
\end{equation}
We obtain the following result.

\begin{lemma}\sl \label{b23}
On $0$-forms, we have
\begin{equation} \label{d8}
L_{\phi} P^{(0)} = P^{(1)} L_{\phi} = L_{\phi} L_{\phi}^{*} L_{\phi} .
\end{equation}
Moreover, for all $\lambda \in \R \setminus \{ 0 \}$, the operator $L_{\phi} : \ker (P ^{(0)} - \lambda ) \longrightarrow \ker (P ^{(1)} - \lambda )$ is injective. Finally, $L_{\phi} ( \ker ( P^{(0)} ) ) = \{ 0 \}$.
\end{lemma}

\begin{proof}
Let us first prove \eqref{d8}. Using \eqref{b1}, \eqref{d9} and the usual cohomology rule (i.e. $d_{\phi , h}^{2} = 0$), we have
\begin{align}
P^{(1)} L_{\phi} &= L_{\phi} L_{\phi}^{*} L_{\phi} + ( Q^{- 1} )^{*} d_{\phi , h}^{*} M d_{\phi , h} Q^{-1} Q d_{\phi , h} a_{h}  \nonumber \\
&= L_{\phi} L_{\phi}^{*} L_{\phi} + ( Q^{- 1} )^{*} d_{\phi , h}^{*} M d_{\phi , h} d_{\phi , h} a_{h} \nonumber \\
&= L_{\phi} L_{\phi}^{*} L_{\phi} = L_{\phi} P^{(0)} .
\end{align}

Now, let $u \neq 0$ be an eigenfunction of $P^{(0)}$ associated to $\lambda \in \R$. In particular, $\Vert L_{\phi} u \Vert^{2} = \lambda \Vert u \Vert^{2}$ vanishes if and only if $\lambda = 0$. Moreover, \eqref{d8} yields
\begin{equation*}
P^{(1)} L_{\phi} u = L_{\phi} P^{(0)} u = \lambda L_{\phi} u .
\end{equation*}
This implies the second part of the lemma.
\end{proof}

We shall now study more precisely the small eigenvalues of $P^{(1)}$. Recall that $\s_{j}$, $j = 2 , \ldots , n_{1} + 1$ denotes the saddle points (of index $1$) of $\phi$. Again we will stick to the analysis already made for the Witten Laplacian on $1$-forms $P^{W , (1)}$ for which we recall the following properties. From Lemma 1.6 and Proposition 1.7 of \cite{HeSj85_01}, the operator $P^{W , (1)}$ is real positive has exactly $n_{1}$ exponentially small (non-zero) eigenvalues (counted with multiplicities). We then recall the construction of associated quasimodes made in Definition 4.3 of \cite{HeKlNi04_01}. Let $u_{j}$ denote a normalized fundamental state of $P^{W , (1)}$ restricted to an appropriated neighborhood of $\s_{j}$ with Dirichlet boundary conditions. The quasimodes $f_{j}^{W , (1)}$ are then defined by
\begin{equation} \label{d20}
f_{j}^{W , (1)} : = \Vert \theta_{j} u_{j} \Vert^{- 1} \theta_{j} (x) u_{j} (x) ,
\end{equation}
where $\theta$ is a well-chosen $C^{\infty}_{0}$ localization function around $\s_{j}$. Since the $f_{j}^{W , (1)}$'s have disjoint support, we immediately deduce
\begin{equation} \label{d21}
\big\< f_{j}^{W , (1)} , f_{j '}^{W , (1)} \big\> = \delta_{j , j '} .
\end{equation}
In particular, the family  $\{ f_{j}^{W , (1)} ; \ j = 2 , \ldots , n_{1} + 1 \}$ is a free family of $1$-forms. Furthermore, Theorem 1.4 of \cite{HeSj85_01} implies that these quasimodes have a WKB writing
\begin{equation} \label{d22}
f_{j}^{W , (1)} (x) = \theta_{j} (x) b^{(1)}_{j} ( x , h ) e^{- \phi_{+ , j} (x) / h} ,
\end{equation}
where $b^{(1)}_{j} ( x , h )$ is a normalization $1$-form having a semiclassical asymptotic, and $\phi_{+ , j}$ is the phase associated to the outgoing manifold of $\xi^{2} + \vert \nabla_x \phi (x) \vert^{2}$ at $( \s_{j} , 0 )$. Moreover, the phase function $\phi_{+ , j}$ satisfies the eikonal equation $\vert \nabla_{x} \phi_{+ , j} \vert^{2} = \vert \nabla_{x} \phi \vert^{2}$ and $\phi_{+ , j} (x) \sim \vert x - \s_{j} \vert^{2}$ near $\s_{j}$. For other properties of $\phi_{+ , j}$ we refer to \cite{HeSj85_01}. On the other hand, Lemma 1.6 and Proposition 1.7 of \cite{HeSj85_01} imply that there exists $\alpha > 0$ independent of $\varepsilon$ such that
\begin{equation} \label{b37}
P^{W , (1)} f_{j}^{W , (1)} = \ooo ( e^{- \alpha / h} ) .
\end{equation}
Eventually, we deduce from Proposition 1.7 of \cite{HeSj85_01} that there exists $\nu > 0$ such that
\begin{equation} \label{b38}
\big\< P^{W , (1)} u , u \big\> \geq \nu h \Vert u \Vert^{2} .
\end{equation}
for all $u \perp \vect \{ f_{j}^{W , (1)} ; \ j = 2 , \ldots , n_{1} + 1 \}$.

Now, let us define the quasimodes associated to $P^{(1)}$ by
\begin{equation} \label{d13}
f_{j}^{(1)} ( x ) : = \beta_{d}^{1 / 2} ( Q^{*} )^{- 1} f_{j}^{W , (1)} ,
\end{equation}
for $2 \leq j \leq n_{1} + 1$. Note that this is possible since $( Q^{*} )^{- 1} \in \Psi^{0} ( \< \xi \> )$. Using that $( Q^{*} )^{- 1}$ is close to $\beta_{d}^{- 1 / 2} \Id$ microlocally near $( \s_{j} , 0 )$, one will prove that they form a good approximately normalized and orthogonal family of quasimodes for $P^{(1)}$.

\begin{lemma}\sl \label{d23}
The system $( f_{j}^{(1)} )_{j}$ is free and for all $j,j'=2,\ldots,n_1+1$ we have
\begin{equation*}
\Vert f_j^{(1)}-f_j^{W,(1)}\Vert=\ooo(h),\qquad\big\< f_{j}^{(1)} , f_{j '}^{(1)} \big\> = \delta_{j , j '} + \ooo (h) \qquad \text{and } \qquad P^{(1)} f_{j}^{(1)}  =  \ooo ( h^{2} ).
\end{equation*}
\end{lemma}

\begin{proof}
From \eqref{d13}, and Corollary \ref{e8} we have 
$$f_j^{(1)}-f_j^{W,(1)}=\big( \beta_d^{\frac 12} ( Q^{*} )^{- 1} - \Id \big) f_{j}^{W , (1)} =\Op( r)f_{j}^{W , (1)}$$
with $r \in S^{0} ( \< \xi \>^{2} )$ such that, modulo $S^{0} ( h \< \xi \>^{2} )$, $r (  x , \xi ) = \ooo ( \vert ( x - \u , \xi ) \vert^{2} )$ near $( \u , 0 )$, $\u \in \uuu$. Moreover, using Taylor expansion and symbolic calculus, we can write
\begin{equation*}
r ( x , \xi ) = \sum_{\vert \alpha + \beta \vert \in \{ 0 , 2  \}} h^{1 - \frac{\vert \alpha + \beta \vert}{2}} r_{\alpha , \beta} ( x , \xi ) ( x - \s_{j} )^{\alpha} \xi^{\beta} ,
\end{equation*}
with $r_{\alpha , \beta} \in S^{0} ( \< \xi \>^{2} )$. Combined with the WKB form of the $f_{j}^{W , (1)}$'s given in \eqref{d22} (and, in particular, with $\phi_{+ , j} (x) \sim \vert x - \s_{j} \vert^{2}$ near $\s_{j}$), it shows that
\begin{equation} \label{d25}
\Op ( r ) f_{j}^{W , (1)} = \ooo ( h ) ,
\end{equation}
which proves the first statement.

The second statement is a direct consequence  the above estimate and \eqref{d21}.

For the last estimate, we follow the same strategy. Thanks to Lemma \ref{d3}, we have
\begin{equation} \label{d24}
P^{(1)} f_{j}^{(1)} = \beta_{d} P^{W , (1)} f_{j}^{W , (1)} + \beta_{d} P^{W , (1)} \big( f_{j}^{(1)} - f_{j}^{W , (1)} \big) + R_{0} f_{j}^{(1)} + R_{1} f_{j}^{(1)} + R_{2} f_{j}^{(1)} .
\end{equation}
Proceeding as above, we write
\begin{equation*}
P^{W , (1)} \big( f_{j}^{(1)} - f_{j}^{W , (1)} \big) = P^{W , (1)} \big( \beta_{d}^{1 /2} ( Q^{*} )^{- 1} - \Id \big) f_{j}^{W , (1)} ,
\end{equation*}
where, using \eqref{d27}, Corollary \ref{e8} and the pseudodifferential calculus, the corresponding operator can be decomposed as
\begin{equation*}
P^{W , (1)} \big( \beta_{d}^{1 /2} ( Q^{*} )^{- 1} - \Id \big) = \Op \bigg( \sum_{\vert \alpha + \beta \vert \in \{ 0 , 2 , 4 \}} h^{2 - \frac{\vert \alpha + \beta \vert}{2}} \widetilde{r}_{\alpha , \beta} ( x , \xi ;h) ( x - \s_{j} )^{\alpha} \xi^{\beta} \bigg) ,
\end{equation*}
for some $\widetilde{r}_{\alpha , \beta} \in S^{0} ( \< \xi \>^{3} )$. Thus, as in \eqref{d25}, we deduce
\begin{equation} \label{d28}
\beta_{d} P^{W , (1)} \big( f_{j}^{(1)} - f_{j}^{W , (1)} \big) = \ooo ( h^{2} ) .
\end{equation}
The same way, we deduce from Lemma \ref{d3} that for any $p=0,1,2$
\begin{equation*}
R_{p} \beta_{d}^{1 / 2} ( Q^{*} )^{- 1} = \Op \bigg( \sum_{\vert \alpha + \beta \vert \in \{ 0 , 2 , 4 \}} h^{2 - \frac{\vert \alpha + \beta \vert}{2}} r^{p}_{\alpha , \beta} ( x , \xi ;h) ( x - \s_{j} )^{\alpha} \xi^{\beta} \bigg) ,
\end{equation*}
with $r^{p}_{\alpha , \beta} \in S^{0} ( \< \xi \>^{3} )$. Thus,
\begin{equation} \label{d29}
R_{p} f_{j}^{(1)} = R_{p} \beta_{d}^{1 / 2} ( Q^{*} )^{- 1} f_{j}^{W , (1)} = \ooo ( h^{2} ) .
\end{equation}
Combining \eqref{d24} together with the estimates \eqref{b37}, \eqref{d28} and \eqref{d29}, we obtain $P^{(1)} f_{j}^{(1)} = \ooo ( h^{2} )$ and concludes the proof of the lemma.
\end{proof}

We shall prove the following proposition which is the analogous of Proposition \ref{b32}.

\begin{proposition}\sl \label{b33}
The operator $P^{(1)}$ has exactly $n_{1}$ $\ooo ( h^{ 2} )$ (real) eigenvalues,
and the remaining part of the spectrum is in  $[ \varepsilon_1 h , + \infty [$, for some $\varepsilon_1 > 0$.
\end{proposition}

The idea of the proof is to consider separately the regions of the phase space closed to the critical points $\uuu$ and away from this set. In the first one, we approximate $P^{(1)}$ by $P^{W , (1)}$ using that $Q \simeq \beta_{d} \Id$ microlocally near $( \u , 0 )$, $\u \in \uuu$. In the second one, we use that (the symbol of) $P^{(1)}$ is elliptic from \eqref{b17}.

We start this strategy with a pseudodifferential IMS formula. For $\eta > 0$ fixed, let $\chi_{0} \in C^{\infty}_{0} ( \R^{2 d} ; [ 0 ,1 ] )$ be supported in a neighborhood of size $\eta$ of $\uuu$ and such that $\chi_{0} = 1$ near $\uuu$ and $\chi_{\infty} : = \sqrt{1 - \chi_{0}^{2}} \in C^{\infty} ( \R^{2 d} )$. In particular,
\begin{equation} \label{b25}
\forall ( x , \xi ) \in \R^{2 d} , \qquad \chi_{0}^{2} ( x , \xi ) + \chi_{\infty}^{2} ( x , \xi ) =1 .
\end{equation}
In the sequel, the remainder terms may depend on $\eta$, but $C$ will denote a positive constant independent on $\eta$ which may change from line to line. Using Lemma \ref{b2} and the shortcut $\Op ( a ) = \Op ( a ) \otimes \Id$, the pseudodifferential calculus gives
\begin{align}
P^{(1)} &= \frac{1}{2} \bigg( \Op \big( \chi_{0}^{2} + \chi_{\infty}^{2} \big) P^{(1)} + P^{(1)} \Op \big( \chi_{0}^{2} + \chi_{\infty}^{2} \big) \bigg)   \nonumber \\
&= \frac{1}{2} \big( \Op ( \chi_{0} )^{2} P^{(1)} + P^{(1)} \Op ( \chi_{0} )^{2} \big) + \frac{1}{2} \big( \Op ( \chi_{\infty} )^{2} P^{(1)} + P^{(1)} \Op ( \chi_{\infty} )^{2} \big) + \Psi^{0} ( h^{2} )   \nonumber \\
&= \Op ( \chi_{0} ) P^{(1)} \Op ( \chi_{0} ) + \Op ( \chi_{\infty} ) P^{(1)} \Op ( \chi_{\infty} ) \nonumber \\
&\qquad \qquad + \frac{1}{2} \big[ \Op ( \chi_{0} ) , \big[ \Op ( \chi_{0} ) , P^{(1)} \big] \big] + \frac{1}{2} \big[ \Op ( \chi_{\infty} ) , \big[ \Op ( \chi_{\infty} ) , P^{(1)} \big] \big] + \ooo ( h^{2} )  \nonumber \\
&= \Op ( \chi_{0} ) P^{(1)} \Op ( \chi_{0} ) + \Op ( \chi_{\infty} ) P^{(1)} \Op ( \chi_{\infty} ) + \ooo ( h^{2} ) . \label{b9}
\end{align}
In the previous estimate, we have crucially used that $\Op ( \chi_{\bullet} ) \otimes \Id$ are matrices of pseudodifferential operators collinear to the identity.

\begin{lemma}\sl \label{b22}
There exists $\delta_{\eta} > 0$, which may depend on $\eta$, such that
\begin{equation} \label{b10}
\Op ( \chi_{\infty} ) P^{(1)} \Op ( \chi_{\infty} ) \geq \delta_{\eta} \Op ( \chi_{\infty} )^{2} + \ooo ( h^{\infty} ) .
\end{equation}
Moreover, there exists $C > 0$ such that, for all $\eta > 0$,
\begin{equation} \label{b14}
\Op ( \chi_{0} ) P^{(1)} \Op ( \chi_{0} ) \geq ( 1 - C \eta  ) \Op ( \chi_{0} ) P^{W , (1)} \Op ( \chi_{0} ) - \big( C \eta h + \ooo ( h^{2} ) \big) .
\end{equation}
\end{lemma}

\begin{proof}
We first estimate $P^{(1)}$ outside of the critical points $\uuu$. Since $\chi_{\infty}$ vanishes near $\uuu$, Lemma \ref{a11} yields that there exist $\delta_{\eta} > 0$ and $\widetilde{p}_{\eta} \in S^{0} (1)$ (which may depend on $\eta$) such that $p = \widetilde{p}_{\eta}$ in a vicinity of the support of $\chi_{\infty}$ and $\widetilde{p}_{\eta} (x , \xi ) \geq 2 \delta_{\eta}$ for all $( x, \xi ) \in \R^{2 d}$. Then, Lemma \ref{b2} and the pseudodifferential calculus (in particular, the G{\aa}rding inequality) imply
\begin{align*}
\Op ( \chi_{\infty} ) P^{(1)} \Op ( \chi_{\infty} ) &= \Op ( \chi_{\infty} ) P^{(0)} \Op ( \chi_{\infty} ) + \Op ( \chi_{\infty} ) \ooo ( h ) \Op ( \chi_{\infty} )    \\
&= \Op ( \chi_{\infty} ) \Op ( \widetilde{p}_{\eta} ) \Op ( \chi_{\infty} ) + \Op ( \chi_{\infty} ) \ooo ( h ) \Op ( \chi_{\infty} ) + \ooo ( h^{\infty} )   \\
&\geq \Op ( \chi_{\infty} ) \big( 2 \delta_{\eta} + \ooo ( h ) \big) \Op ( \chi_{\infty} ) + \ooo ( h^{\infty} ) ,
\end{align*}
which implies \eqref{b10} for $h$ small enough. Here, we have identify as before $A$ with $A \otimes \Id$ for scalar operators $A$.

We now consider $\Op ( \chi_{0} ) P^{(1)} \Op ( \chi_{0} )$ . Thanks to Lemma \ref{d3}, we can write
\begin{equation*}
\Op ( \chi_{0} ) P^{(1)} \Op ( \chi_{0})=\Op ( \chi_{0} ) P^{W,(1)} \Op ( \chi_{0})+ \sum_{k=0}^2\Op ( \chi_{0} ) R_k \Op ( \chi_{0}) .
\end{equation*}
Let $\widetilde{\chi}_{0} \in C^{\infty}_{0} ( \R^{2 d} ; [ 0 ,1 ] )$ be supported in a neighborhood of size $\eta$ of $( \u , 0 )$, $\u \in \uuu$, and such that $\widetilde{\chi}_{0} = 1$ near the support of $\chi_{0}$. Then, for $\omega \in \Omega^{1} ( \R^{d} )$, $\< R_{0} \Op ( \chi_{0} ) \omega , \Op ( \chi_{0} ) \omega \>$ is a finite sum of terms of the form
\begin{equation} \label{b18}
\widetilde{r}_{0} = \big\< ( d_{\phi , h}^{a} )^{*} \big( \Op ( r_{0} ) + \Psi^{0} (h) \big) d_{\phi , h}^{b} \Op ( \chi_{0} ) \omega_{j} , \Op ( \chi_{0} ) \omega_{k} \big\> .
\end{equation}
Using functional analysis and pseudodifferential calculus, we get
\begin{align}
\vert \widetilde{r}_{0} \vert &= \big\vert \big\< \big( \Op ( r_{0} \widetilde{\chi}_{0} ) + \Psi^{0} (h) \big) d_{\phi , h}^{b} \Op ( \chi_{0} ) \omega_{j} , d_{\phi , h}^{a} \Op ( \chi_{0} ) \omega_{k} \big\> \big\vert + \ooo ( h^{\infty} ) \Vert \omega \Vert^{2}   \nonumber \\
&\leq \big( \big\Vert \Op ( r_{0} \widetilde{\chi}_{0} ) \big\Vert + \ooo (h) \big) \big\Vert d_{\phi , h}^{b} \Op ( \chi_{0} ) \omega_{j} \big\Vert \big\Vert d_{\phi , h}^{a} \Op ( \chi_{0} ) \omega_{k} \big\Vert + \ooo ( h^{\infty} ) \Vert \omega \Vert^{2}   \nonumber \\
&\leq \big( \big\Vert \Op ( r_{0} \widetilde{\chi}_{0} ) \big\Vert + \ooo (h) \big) \big\< P^{W , (0)} \Op ( \chi_{0} ) \omega , \Op ( \chi_{0} ) \omega \big\> + \ooo ( h^{\infty} ) \Vert \omega \Vert^{2} .  \label{b11}
\end{align}
Recall now that, for $a \in S^{0} (1)$,
\begin{equation*}
\big\Vert \Op (a) \big\Vert_{L^{2} ( \R^{d} ) \to L^{2} ( \R^{d} )} = \Vert a \Vert_{L^{\infty} ( \R^{2 d} )} + \ooo (h) ,
\end{equation*}
(see e.g. Zworski \cite[Theorem 13.13]{Zw12_01}). Thus, using that $\widetilde{\chi}_{0}$ is supported in a neighborhood of size $\eta$ of
$( \u , 0 )$ at which $r_{0}$ vanishes, it yields $\Vert \Op ( r_{0} \widetilde{\chi}_{0} ) \Vert \leq C \eta$ and \eqref{b11} implies
\begin{equation} \label{b12}
\big\vert \< R_{0} \Op ( \chi_{0} ) \omega , \Op ( \chi_{0} ) \omega \> \big\vert
\leq C \eta \big\< P^{W , (0)} \Op ( \chi_{0} ) \omega , \Op ( \chi_{0} ) \omega \big\> + \ooo ( h^{\infty} ) \Vert \omega \Vert^{2} .
\end{equation}

As before, $\< R_{1} \Op ( \chi_{0} ) \omega , \Op ( \chi_{0} ) \omega \>$ is a finite sum of terms of the form
\begin{equation} \label{b21}
\widetilde{r}_{1} = \big\< \Psi^{0} (h) d_{\phi , h}^{a} \Op ( \chi_{0} ) \omega_{j} , \Op ( \chi_{0} ) \omega_{k} \big\> ,
\end{equation}
or its complex conjugate. These terms can be estimate as
\begin{align*}
\vert \widetilde{r}_{1} \vert &\leq C h \big\Vert d_{\phi , h}^{a} \Op ( \chi_{0} ) \omega_{j} \big\Vert \Vert \omega \Vert    \\
&\leq \eta \big\Vert d_{\phi , h}^{a} \Op ( \chi_{0} ) \omega_{j} \big\Vert^{2} + \ooo ( h^{2} ) \Vert \omega \Vert^{2}    \\
&\leq \eta \big\< P^{W , (0)} \Op ( \chi_{0} ) \omega , \Op ( \chi_{0} ) \omega \big\> + \ooo ( h^{2} ) \Vert \omega \Vert^{2} ,
\end{align*}
and then
\begin{equation} \label{b13}
\big\vert \< R_{1} \Op ( \chi_{0} ) \omega , \Op ( \chi_{0} ) \omega \> \big\vert
\leq C \eta \big\< P^{W , (0)} \Op ( \chi_{0} ) \omega , \Op ( \chi_{0} ) \omega \big\> + \ooo ( h^{2} ) \Vert \omega \Vert^{2} .
\end{equation}
Combining Lemma \ref{d3} with the estimates \eqref{b12}, \eqref{b13} and $R_{2} \in \Psi^{0} ( h^{2} )$, we obtain
\begin{equation*}
\Op ( \chi_{0} ) P^{(1)} \Op ( \chi_{0} ) \geq \Op ( \chi_{0} ) P^{W , (1)} \Op ( \chi_{0} ) - C \eta \Op ( \chi_{0} ) P^{W , (0)} \Op ( \chi_{0} ) - \ooo ( h^{2} ) .
\end{equation*}
Since $P^{W , (1)} = P^{W , (0)} \otimes \Id + \Psi^{0} (h)$ (see Equation (1.9) of \cite{HeSj85_01} for example), this inequality gives \eqref{b14}.
\end{proof}

Let $\Pi$ denote the orthogonal projection onto $\vect \{ f^{(1)}_{j} ; \ j = 2 , \ldots , n_{1} + 1 \}$. Using the previous lemma and its proof, we can describe the action of $P^{(1)}$ on $\Pi$:

\begin{lemma}\sl \label{b20}
The rank of $\Pi$ is $n_{1}$ for $h$ small enough. Moreover,
\begin{equation} \label{b29}
P^{(1)} \Pi = \ooo ( h^{2} ) \qquad \text{and} \qquad \Pi P^{(1)} = \ooo ( h^{2} ) .
\end{equation}
Finally, there exists $\varepsilon_1 > 0$ such that
\begin{equation} \label{b30}
( 1 - \Pi ) P^{(1)} ( 1 - \Pi ) \geq \varepsilon_1 h( 1 - \Pi ) ,
\end{equation}
for $h$ small enough.
\end{lemma}

\begin{proof}
Since the functions $f^{(1)}_{j}$'s are almost orthogonal (i.e. $\< f^{(1)}_{j} , f^{(1)}_{j '} \> = \delta_{j , j '} + \ooo ( h )$), the rank of $\Pi$ is $n_{1}$. Moreover, \eqref{b29} is a direct consequence of Lemma \ref{d23}.

We now give the lower bound for $P^{(1)}$ on the range of $1 - \Pi$. 
Let $\mathcal{E}^{(1)}$ denote the space spanned by the $f_k^{W,(1)},k=2,\ldots,n_1+1$ and $\mathcal{F}^{(1)}$ denote the eigenspace associated to the $n_1$ first eigenvalues of $P^{W,(1)}$. Let $\Pi_{\mathcal E^{(1)}},\,\Pi_{\mathcal F^{(1)}}$ denote the corresponding orthogonal projector. It follows from \cite{HeSj85_01} that 
$\Vert \Pi_{\mathcal E^{(1)}}-\Pi_{\mathcal F^{(1)}}\Vert=\ooo(e^{-c/h})$ for some $c>0$. On the other hand, it follows from the first estimate of Lemma \ref{d23} that 
$\Vert\Pi-\Pi_{\mathcal E^{(1)}}\Vert=\ooo(h)$. Combining these two estimates, we get 
$$
\Vert\Pi-\Pi_{\mathcal F^{(1)}}\Vert=\ooo(h).
$$
Using this estimate and the spectral properties of $P^{W, (1)}$, we get
\begin{equation}\label{eq:minor_PW1}
P^{W , (1)} \geq \nu h - \nu h \Pi _{\mathcal F^{(1)}}\geq\nu h - \nu h \Pi +\ooo(h^2),
\end{equation}
for some $\nu > 0$. From \eqref{d22} and integration by parts, we also have $\Op ( \chi_{0} ) \Pi = \Pi + \ooo ( h^{\infty} )$. 
Estimate  \eqref{eq:minor_PW1} together with \eqref{b25}, \eqref{b9}, \eqref{b10} and \eqref{b14} give
\begin{align}
P^{(1)} &= \Op ( \chi_{0} ) P^{(1)} \Op ( \chi_{0} ) + \Op ( \chi_{\infty} ) P^{(1)} \Op ( \chi_{\infty} ) + \ooo ( h^{2} )    \nonumber \\
&\geq (1 - C \eta ) \Op ( \chi_{0} ) P^{W , (1)} \Op ( \chi_{0} ) + \delta_{\eta} \Op ( \chi_{\infty} )^{2} - \big( C \eta h + \ooo ( h^{2} ) \big)  \nonumber \\
&\geq \nu h (1 - C \eta ) \Op ( \chi_{0} )^{2} - \nu h (1 - C \eta ) \Pi + \delta_{\eta} \Op ( \chi_{\infty} )^{2} - \big( C \eta h + \ooo ( h^{2} ) \big)  \nonumber  \\
&\geq \nu h (1 - C \eta ) - \nu h (1 - C \eta ) \Pi - \big( C \eta h + \ooo ( h^{2} ) \big) .
\end{align}
Thus, taking $\eta > 0$ small enough and applying $1 - \Pi$, we eventually obtain \eqref{b30} for some $\varepsilon_1>0$.
\end{proof}

\begin{proof}[Proof of Proposition \ref{b33}]
From Proposition \ref{a47} and Lemma \ref{b2}, the operator $P^{(1)}$ is bounded and its essential spectrum is
above some positive constant independent of $h$. Next, the maxi-min principle together with \eqref{b29} implies that
$P^{(1)}$ has at least $\rank ( \Pi ) = n_{1}$ eigenvalues below $C h^{ 2}$.
The same way, \eqref{b30} yields that $P^{(1)}$ has at most $n_{1}$ eigenvalues below $\varepsilon_1 h$. Eventually,
\begin{equation*}
P^{(1)} = ( 1 - \Pi ) P^{(1)} ( 1 - \Pi ) + \Pi P^{(1)} ( 1 - \Pi ) + ( 1 - \Pi ) P^{(1)} \Pi + \Pi P^{(1)} \Pi \geq - C h^{2} ,
\end{equation*}
proves that all the spectrum of $P^{(1)}$ is above $- C h^{2}$.
\end{proof}

\section{Eigenspace analysis and proof of the main Theorem} \label{z4}

Now we want to project the preceding quasimodes onto the generalized eigenspaces associated to exponentially small eigenvalues, and prove the main theorem. Recall that we have built in the preceding section quasimodes $f_k^{(0)}$, $k= 1, \ldots , n_0$, for $P^{(0)}$ with good support properties. To each quasimode we will  associate a function in  $E^{(0)}$, the eigenspace associated to $\ooo(h^{2})$ eigenvalues. For this, we first define the spectral projector
\begin{equation} \label{d4}
\Pi^{(0)} = \frac{1}{2\pi i} \int_\gamma (z-P^{(0)})^{-1} d z ,
\end{equation}
where $\gamma = \partial B ( 0 , \varepsilon_0 h / 2 )$ and  $\varepsilon_0 > 0$ is defined in Proposition \ref{b32}. From the  fact that $P^{(0)}$ is selfadjoint, we get that
\begin{equation*}
\Pi^{(0)} = \ooo ( 1 ) .
\end{equation*}
For the following, we denote the corresponding projection
\begin{equation*}
e_k^{(0)} = \Pi^{(0)} ( f_k^{(0)} ) .
\end{equation*}
We have then

\begin{lemma}\sl \label{d5}
The system $( e_k^{(0)} )_{k}$ is free and spans $E^{(0)}$. Besides, there exists $\alpha > 0$ independent of $\varepsilon$ such that
\begin{equation*}
e_k^{(0)} = f_k^{(0)} + \ooo ( e^{- \alpha / h} ) \qquad \text{and} \qquad \big\< e_{k}^{(0)} , e_{k '}^{(0)} \big\> = \delta_{k , k '} + \ooo ( h ) .
\end{equation*}
\end{lemma}

\begin{proof}
The proof follows \cite{HeSj85_01} (see also \cite{DiSj99_01}). We sketch it for completeness sake and to give the necessary modifications. Using \eqref{d4} and the Cauchy formula, we therefore get
\begin{align*}
e_k^{(0)} - f_k^{(0)}
& = \Pi^{(0)} f_k^{(0)} - f_k^{(0)} \\
& = \frac{1}{2 \pi i } \int_\gamma (z - P^{(0)})^{-1} f_k^{(0)} dz - \frac{1}{2 \pi i } \int_\gamma z^{-1} f_k^{(0)}d z \\
& = \frac{1}{2 \pi i } \int_\gamma (z - P^{(0)})^{-1} z^{-1} P^{(0)} f_k^{(0)} d z .
\end{align*}
Since $P^{(0)}$ is selfadjoint and according to Proposition \ref{b32}, we have
\begin{equation*}
\big\Vert ( z - P^{(0)} )^{- 1} \big\Vert = \ooo ( h^{- 1} ) ,
\end{equation*}
uniformly for $z \in \gamma$. Using also the second estimate in Lemma \ref{d2}, it yields
\begin{equation*}
\big\Vert ( z - P^{(0)})^{- 1} z^{- 1} P^{(0)} f_k^{(0)} \big\Vert = \ooo \big( h^{- 2} e^{- \alpha / h} \big) ,
\end{equation*}
and eventually after integration
\begin{equation*}
\big\Vert e_k^{(0)} - f_k^{(0)} \big\Vert = \ooo \big( h^{- 1} e^{- \alpha / h} \big) .
\end{equation*}
Decreasing $\alpha$, we obtain the first estimate of the lemma. In particular, this implies that the family $( e_k^{(0)} )_{k}$ is free. Using that $E^{(0)}$ is of dimension $n_0$, the family $( e_k^{(0)} )_{k}$ spans $E^{(0)}$.

For the last equality of the lemma, we just have to notice that
\begin{equation*}
\big\< e_k^{(0)} , e_{k '}^{(0)} \big\> = \< f_{k}^{(0)} , f_{k'}^{(0)} \big\> + \ooo( e ^{- \alpha/h} ) = \delta_{k , k '} + \ooo(h) + \ooo( e ^{- \alpha / h} ) = \delta_{k , k '} + \ooo ( h ) ,
\end{equation*}
according to Lemma \ref{d2}. The proof is complete.
\end{proof}

We can do a similar study for the analysis of $P^{(1)}$, for which we know that exactly $n_1$ (real) eigenvalues are $\ooo ( h^{2} )$, and among them at least $n_{0} - 1$ are exponentially small. Note that there is no particular reasons for the remaining ones to be also exponentially small.

To the family of quasimodes $( f_{j}^{(1)})_{j}$, we now associate a family of functions in $E^{(1)}$, the eigenspace associated to $\ooo ( h^{2} )$ eigenvalues for ${P}^{(1)}$. Thanks to the spectral properties of the selfadjoint operator $P^{(1)}$, its spectral projector onto $E^{(1)}$ is given by
\begin{equation} \label{d11}
\Pi^{(1)} = \frac{1}{2 \pi i} \int_\gamma ( z - P^{(1)} )^{- 1} d z ,
\end{equation}
where $\gamma = \partial B ( 0 , \varepsilon_1 h / 2)$ where $\varepsilon_1$ is defined in Proposition \ref{b33}. In the sequel, we denote
\begin{equation*}
e_j^{(1)} = \Pi^{(1)} ( f_{j}^{(1)} ) .
\end{equation*}
Mimicking the proof of Lemma \ref{d5}, one can show that the family $(e_j^{(1)})_{j}$ satisfies the following estimates:

\begin{lemma}\sl \label{d6}
The system $( e_j^{(1)} )_{j}$ is free and spans $E^{(1)}$. Besides, we have
\begin{equation*}
e_j^{(1)} = f_j^{(1)} + \ooo ( h ) , \qquad \text{and} \qquad \big\< e_{j}^{(1)} , e_{j '}^{(1)} \big\> = \delta_{j , j '} + \ooo ( h ) .
\end{equation*}
\end{lemma}

Thanks to the preceding lemmas, the families $( e_{k}^{(0)} )_{k}$ and $( e_{j}^{(1)} )_{j}$ are orthonormal, apart from an $\ooo ( h )$ factor. For computing accurately the eigenvalues of $P^{(0)}$ and prove the main theorem, we need more precise estimates of exponential type. For this, we will use the intertwining relation $L_{\phi} P^{(0)} = P^{(1)} L_{\phi}$.

More precisely, we denote by $L$ the $n_{1} \times n_{0}$ matrix of this restriction of $L_\phi$ with respect to the basis $( e_j^{(1)} )_{j}$ and $( e_k^{(0)} )_{k}$:
\begin{equation}
L_{j , k} : = \big\< e_{j}^{(1)} , L_{\phi} e_{k}^{(0)} \big\> .
\end{equation}
The classical way (\cite{HeKlNi04_01}, \cite{HeSj85_01}, \ldots) of computing the exponentially small eigenvalues of $P^{(0)}$ is then to compute accurately the  singular values of $L$.
For this we first state a refined lemma about exponential estimates.

\begin{lemma}\sl \label{d7}
There exist $\alpha > 0$ independent of $\varepsilon$ such that
\begin{equation} \label{d32}
L_{\phi} L_{\phi}^{*} f_j^{(1)} = \ooo ( e^{- \alpha / h} ) ,
\end{equation}
and also a smooth $1$-form $r_j^{(1)}$ such that
\begin{equation*}
L_{\phi}^{*} \big( e_{j}^{(1)} - f_{j}^{(1)} \big) = L_{\phi}^{*} r_{j}^{(1)} \qquad \text{and} \qquad r_{j}^{(1)} = \ooo ( e^{- \alpha / h} ) .
\end{equation*}
\end{lemma}

\begin{proof}
We first note that
\begin{align}
L_{\phi} L_{\phi}^{*} f_{j}^{(1)} &=  L_{\phi} a_h d_{\phi , h}^{*} Q^{*} ( Q^{*} )^{-1} f_{j}^{W , (1)}    \nonumber  \\
&= L_{\phi} a_{h} \big( d_{\phi , h}^{*} f_{j}^{W , (1)} \big) . \label{d33}
\end{align}
On the other hand, \eqref{d27} and \eqref{b37} give
\begin{equation*}
\big\Vert d_{\phi , h}^{*} f_{j}^{W , (1)} \big\Vert^{2} \leq \big\Vert d_{\phi , h}^{*} f_{j}^{W , (1)} \big\Vert^{2} + \big\Vert d_{\phi , h} f_{j}^{W , (1)} \big\Vert^{2} = \big\< P^{(1)} f_{j}^{W , (1)} , f_{j}^{W , (1)} \big\> = \ooo ( e^{- \alpha / h} ) .
\end{equation*}
for some $\alpha > 0$ independent of $\varepsilon$. Since $a_{h}$ and $L_{\phi}$ are uniformly bounded operators, \eqref{d33} provides the required estimate.

Now we show the second and third equalities, following closely the proof of Lemma \ref{d5}. Using \eqref{d4}, the intertwining relation (see Lemma \ref{b23}) and the Cauchy formula, we have
\begin{align}
L_{\phi}^{*} \big( e_{j}^{(1)} - f_{j}^{(1)} \big) &= L_{\phi}^{*} \Pi^{(1)} f_{j}^{(1)} -  L_{\phi}^{*} f_{j}^{(1)}   \nonumber \\
&= \Pi^{(0)} L_{\phi}^{*} f_{j}^{(1)} - L_{\phi}^{*} f_{j}^{(1)}   \nonumber \\
&= \frac{1}{2 \pi i} \int_{\gamma} ( z - P^{(0)} )^{- 1} L_{\phi}^{*} f_{j}^{(1)} d z - \frac{1}{2 \pi i} \int_{\gamma} z^{-1} L_{\phi}^{*} f_{j}^{(1)} d z  \nonumber \\
 & = \frac{1}{2 \pi i } \int_{\gamma} (z - P^{(0)})^{-1} z^{-1} P^{(0)} L_{\phi}^{*} f_{j}^{(1)} d z ,
\end{align}
where $\gamma = \partial B ( 0 , \min ( \varepsilon_{0} , \varepsilon_{1} ) h / 2 )$. Using again Lemma \ref{b23}, it becomes
\begin{align*}
L_{\phi}^{*} \big( e_{j}^{(1)} - f_{j}^{(1)} \big) &= \frac{1}{2 \pi i} \int_{\gamma} ( z - P^{(0)} )^{- 1} z^{- 1} L_{\phi}^{*} L_{\phi} L_{\phi}^{*} f_{j}^{(1)} d z    \nonumber \\
&= L_{\phi}^{*} \frac{1}{2 \pi i} \int_{\gamma} (z - P^{(1)} )^{- 1} z^{-1}  L_{\phi} L_{\phi}^{*} f_{j}^{(1)} d z .
\end{align*}
We then pose
\begin{equation}
r_{j}^{(1)} = \Big( \frac{1}{2 \pi i} \int_{\gamma} ( z - P^{(1)} )^{- 1} z^{-1} d z \Big) L_{\phi} L_{\phi}^{*} f_{j}^{(1)} ,
\end{equation}
and the preceding equality reads
\begin{equation}
L_{\phi}^{*} \big( e_{j}^{(1)} - f_{j}^{(1)} \big) =  L_{\phi}^{*} r_{j}^{(1)} .
\end{equation}
Moreover, as in proof of Lemma \ref{d5}, we have
\begin{equation*}
\frac{1}{2 \pi i} \int_{\gamma} ( z - P^{(1)} )^{- 1} z^{- 1} d z = \ooo ( h^{-1} ) .
\end{equation*}
Combining with \eqref{d32}, it shows that $r_{j}^{(1)} = \ooo ( e^{- \alpha / h} )$ for some (new) $\alpha > 0$.
\end{proof}

We begin the study of the matrix $L$ by the following lemma.

\begin{lemma}\sl \label{d14}
There exists $\alpha ' > 0$ such that, if $\varepsilon > 0$ is sufficiently small and fixed, we have, for all $2 \leq j \leq n_{1} + 1$ and $2 \leq k \leq n_{0}$,
\begin{equation*}
L_{j , k} = \big\< f_{j}^{(1)}, L_{\phi} f_{k}^{(0)} \big\> + \ooo \big( e^{- ( S_{k} + \alpha ' ) / h} \big) .
\end{equation*}
Moreover, $L_{j , 1} = 0$ for all $2 \leq j \leq n_{1} + 1$.
\end{lemma}

\begin{proof}
We first treat the case $k = 1$. Since $f_{1}^{(0)}$ is collinear to $a_{h}^{- 1} e^{- \phi / h}$, it belongs to $\ker ( P^{(0)} )$. Then, $e_{1}^{(0)} = \Pi^{(0)} f_{1}^{(0)} = f_{1}^{(0)}$ satisfies $L_{\phi} e_{1}^{(0)} = 0$ from Lemma \ref{b23}. In particular, $L_{j , 1} = 0$ for all $2 \leq j \leq n_{1} + 1$.

We now assume $2 \leq k \leq n_{0}$. Using Lemma \ref{b23} and the definition of $e^{( \star )}_{\bullet}$, we can write
\begin{align*}
L_{j , k} & = \big\< e_{j}^{(1)}, L_{\phi} e_{k}^{(0)} \big\> = \big\< e_{j}^{(1)}, L_{\phi} \Pi^{(0)} f_{k}^{(0)} \big\> = \big\< e_{j}^{(1)}, \Pi^{(1)} L_{\phi} f_{k}^{(0)} \big\>   \\
&= \big\< \Pi^{(1)} e_{j}^{(1)}, L_{\phi} f_{k}^{(0)} \big\> = \big\< e_{j}^{(1)}, L_{\phi} f_{k}^{(0)} \big\> = \big\< f_{j}^{(1)}, L_{\phi} f_{k}^{(0)} \big\> + \big\< e_{j}^{(1)}- f_{j}^{(1)}, L_{\phi} f_{k}^{(0)} \big\>   \\
&= \big\< f_{j}^{(1)}, L_{\phi} f_{k}^{(0)} \big\> + \big\< L_{\phi}^{*} (e_{j}^{(1)}- f_{j}^{(1)}) , f_{k}^{(0)} \big\> .
\end{align*}
From Lemma \ref{d7}, it becomes
\begin{align}
L_{j,k} &= \big\< f_{j}^{(1)}, L_{\phi} f_{k}^{(0)} \big\> + \big\< L_{\phi}^{*} r_{j}^{(1)} , f_{k}^{(0)}\big\>    \nonumber \\
&= \big\< f_{j}^{(1)}, L_{\phi} f_{k}^{(0)} \big\> + \big\< r_{j}^{(1)} , L_{\phi} f_{k}^{(0)} \big\> . \label{d30}
\end{align}
Now, since $Q$ is bounded and according to Lemma \ref{e7}, we have
\begin{equation*}
L_{\phi} f_{k}^{(0)} = Q  \d_{\phi,h}f^{W,(0)}= \ooo \big( e^{- ( S_{k} - C \varepsilon ) / h } \big).
\end{equation*}
Using Lemma \ref{d7} again, it yields
\begin{equation} \label{d31}
\big\< r_{j}^{(1)} , L_{\phi} f_{k}^{(0)} \big\> = \ooo \big( e^{- ( S_{k} + \alpha - C \varepsilon ) / h } \big) ,
\end{equation}
with $\alpha > 0$ independent of $\varepsilon$. Eventually, taking $\varepsilon > 0$ small enough, the lemma follows from \eqref{d30} and \eqref{d31}.
\end{proof}

Now we just recall the explicit computation of the matrix $L$. This is just a consequence of the study of the corresponding Witten Laplacian.

\begin{lemma}\sl \label{d12}
For all $2 \leq j \leq n_{1} + 1$ and $2 \leq k \leq n_{0}$, we have
\begin{equation*}
L_{k , k} = \Big( \frac{h}{( 2 d + 4) \pi} \Big)^{1 / 2} \mu_{k}^{1 / 2} \bigg\vert \frac{\det \phi '' ( \m_{k} )}{\det \phi '' ( \s_{k} )} \bigg\vert^{1 / 4} e^{- S_{k} / h} ( 1 + \ooo ( h ) ) = : h^{1 / 2} \ell_{k} ( h ) e^{- S_{k} / h} ,
\end{equation*}
and
\begin{equation*}
\forall j \neq k , \qquad  L_{j , k} = \ooo \big( e^{- ( S_{k} + \alpha ' ) / h} \big) ,
\end{equation*}
where $S_{k} : = \phi ( \s_{k} ) - \phi ( \m_{k} )$ and $- \mu_{k}$ denotes the unique negative eigenvalue of $\phi ''$ at $\s_{k}$.
\end{lemma}

\begin{proof}
First, we note that
\begin{equation*}
\big\< f_{j}^{(1)} , L_{\phi} f_{k}^{(0)} \big\> = \beta_d^{1 / 2} \big\< f_{j}^{W , (1)} , d_{\phi , h} f_{k}^{W , (0)} \big\> ,
\end{equation*}
by \eqref{b35}, \eqref{d13} and $L_{\phi} = Q d_{\phi , h} a_{h}$. Thus, Lemma \ref{d14} implies
\begin{equation*}
L_{j , k} = \beta_{d}^{1 / 2} \big\< f_{j}^{W , (1)} , d_{\phi , h} f_{k}^{W , (0)} \big\> + \ooo \big( e^{- ( S_{k} + \alpha ' ) / h} \big) .
\end{equation*}
The first term is exactly the approximate singular value of $d_{\phi , h}$ computed in \cite{HeKlNi04_01}. The result is then a direct consequence of Proposition 6.4 of \cite{HeKlNi04_01}.
\end{proof}

Now we are able to compute the singular values of $L$ (i.e. the eigenvalues of $( L^{*} L )^{1 / 2}$).

\begin{lemma}\sl \label{d15}
There exists $\alpha ' > 0$ such that the singular values $\nu_{k} ( L )$ of $L$, enumerated in a a suitable order, satisfy
\begin{equation*}
\forall 1 \leq k \leq n_{0} , \qquad \nu_{k} ( L ) = \vert L_{k , k} \vert \big( 1 + \ooo ( e^{- \alpha ' / h} ) \big) .
\end{equation*}
\end{lemma}

\begin{proof}
Since the first column of $L$ consists of zeros, we get $\nu_{1} = 0$. Moreover, the other singular values of $L$ are those of the reduced matrix $L '$ with entries $L_{j , k} ' = L_{j + 1 , k + 1}$ for $1 \leq j \leq n_{1}$ and $1 \leq k \leq n_{0} - 1$. We shall now use that the dominant term in each column of $L '$ lies on the diagonal. Define the $( n_{0} - 1 ) \times ( n_{0} - 1 )$ diagonal matrix $D$ by
\begin{equation*}
D : = \diag \big( L_{k + 1 , k + 1} ; \ k = 1, \ldots , n_{0} - 1 \big) .
\end{equation*}
Notice that $D$ is invertible, thanks to the ellipticity of $\ell_{k + 1} ( h )$, and that $\nu_{k} ( D ) = \vert L_{k + 1 , k + 1} \vert$. We also define the $n_{1} \times ( n_{0} - 1 )$ characteristic matrix of $L '$
\begin{equation*}
U = ( \delta_{ j , k} )_{j , k} .
\end{equation*}
From Lemma \ref{d12}, there is a constant $\alpha ' >0 $ such that
\begin{equation} \label{d17}
L ' = \big( U + \ooo ( e^{- \alpha ' / h } ) \big) D .
\end{equation}
The Fan inequalities (see for example Theorem 1.6 of \cite{Si79_01}) therefore give
\begin{equation} \label{d16}
\nu_{k} ( L ' ) \leq \big( 1 + \ooo( e^{- \alpha ' / h} ) \big) \nu_{k} (D) .
\end{equation}
To get the opposite estimate, we remark that $U^{*} U = \Id_{n_{0} - 1}$. Then, \eqref{d17} implies
\begin{equation*}
D = \big( 1 + \ooo ( e^{- \alpha ' / h} ) \big ) U^{*} L ' ,
\end{equation*}
and, as before,
\begin{equation} \label{d18}
\nu_{k} ( D ) \leq \big( 1 + \ooo ( e^{- \alpha ' / h} ) \big) \nu_{k} ( L ' ) .
\end{equation}
The lemma follows from $\nu_{k + 1} ( L ) = \nu_{k} ( L ' )$, \eqref{d16}, \eqref{d18} and $\nu_{k} ( D ) = \vert L_{k + 1 , k + 1} \vert$.
\end{proof}

Now, Theorem \ref{e3} is a direct consequence of the explicit computations of Lemma \ref{d12} and of the following equivalent formulation.

\begin{lemma}\sl
The non-zero exponentially small eigenvalues of $P_{h}$ are of the form
\begin{equation*}
h \big( \ell_{k}^{2} (h) + \ooo ( h ) \big) e^{- 2 S_{k} / h} ,
\end{equation*}
for $2 \leq k \leq n_{0}$.
\end{lemma}

\begin{proof}
According to Lemma \ref{d5} and Lemma \ref{d6}, the bases $( e_{k}^{(0)} )_{k}$ and $( e_{j}^{(1)} )_{j}$ of $E^{(0)}$ and $E^{(1)}$ respectively are orthonormal up to $\ooo(h)$ small errors. Let $( \widetilde{e}_{k}^{(0)} )_{k}$ and $( \widetilde{e}_{j}^{(1)} )_{j}$ be the corresponding orthonormalizations (obtained by taking square roots of the Gramians), which differ from the original bases by $\ooo(h)$ small recombinations. Then, with respect to the new bases, the matrix of $L_\phi $ take the form $\widetilde{L} = ( 1 + \ooo ( h ) ) ) L ( 1 + \ooo ( h ) )$. Using the Fan inequalities, we see that the conclusion of Lemma \ref{d15} is also valid for $\widetilde{L}$ (note that there is no need to have
exponentially small errors here). Since the matrix of the restriction of $P^{(0)}$ to $E^{(0)}$ with respect to the basis $( \widetilde{e}_{k}^{(0)} )_{k}$ is given by $\widetilde{L}^{*} \widetilde{L}$, the lemma follows.
\end{proof}

We end this part by showing that the mains theorems stated in Section \ref{z6} imply the metastability of the system.

\begin{proof}[Proof of Corollary \ref{z5}]
We first prove \eqref{z3} and \eqref{z1}. If $\phi$ has a unique minimum, Theorem \ref{a1} gives
\begin{equation*}
\big\Vert ( \T_{h}^{\star} )^{n} ( d \nu_{h} ) - d \nu_{h , \infty} \big\Vert_{\hhh_h} \leq ( 1 - \delta h )^{n} \Vert d \nu_{h} \Vert_{\hhh_h} = e^{n \ln ( 1 - \delta h ) + \vert \ln h \vert} h \Vert d \nu_{h} \Vert_{\hhh_h} .
\end{equation*}
Using that $n \ln ( 1 - \delta h ) \sim - \delta h n$, this estimate yields
\begin{equation*}
\big\Vert ( \T_{h}^{\star} )^{n} ( d \nu_{h} ) - d \nu_{h , \infty} \big\Vert_{\hhh_h} \leq h \Vert d \nu_{h} \Vert_{\hhh_h} ,
\end{equation*}
for $n \gtrsim \vert \ln h \vert h^{-1}$. The same way, if $\phi$ has several minima, Theorem \ref{e3} implies
\begin{equation*}
\big\Vert ( \T_{h}^{\star} )^{n} ( d \nu_{h} ) - d \nu_{h , \infty} \big\Vert_{\hhh_h} \leq ( \lambda_{2}^{\star} (h) )^{n} \Vert d \nu_{h} \Vert_{\hhh_h} = e^{n \ln ( \lambda_{2}^{\star} (h) ) + \vert \ln h \vert} h \Vert d \nu_{h} \Vert_{\hhh_h} .
\end{equation*}
Using now $n \ln ( \lambda_{2}^{\star} (h) ) \sim n ( \lambda_{2}^{\star} (h) - 1 ) \sim - C n h e^{- S_{2} / h}$, for some $C > 0$, this estimate yields
\begin{equation*}
\big\Vert ( \T_{h}^{\star} )^{n} ( d \nu_{h} ) - d \nu_{h , \infty} \big\Vert_{\hhh_h} \leq h \Vert d \nu_{h} \Vert_{\hhh_h} ,
\end{equation*}
for $n \gtrsim \vert \ln h \vert h^{-1} e^{S_{2} / h}$.

It remains to show \eqref{z2}. From Theorem \ref{a1}, Theorem \ref{e3} and the proof of \eqref{z3}, we can write
\begin{equation*}
( \T_{h}^{\star} )^{n} ( d \nu_{h} ) =  \sum_{k = 1}^{n_{0}} ( \lambda_{k}^{\star} (h) )^{n} \Pi_{k} d \nu_{h} + \ooo (h) \Vert d \nu_{h} \Vert_{\hhh_h} ,
\end{equation*}
for $n \gtrsim \vert \ln h \vert h^{-1}$. Here, $\Pi_{k}$ is the spectral projector of $\T_{h}^{\star}$ associated to the eigenvalue $\lambda_{k}^{\star} (h)$. If we assume in addition that $n \lesssim e^{2S_{n_0} / h}$, then $( \lambda_{k}^{\star} (h) )^{n} = 1 + \ooo (h)$ for any $k=1,\ldots,n_0$. Thus, the previous equation becomes
\begin{equation} \label{z7}
( \T_{h}^{\star} )^{n} ( d \nu_{h} ) =  \Pi^{(0)} d \nu_{h} + \ooo (h) \Vert d \nu_{h} \Vert_{\hhh_h} ,
\end{equation}
since $\Pi^{(0)} = \Pi_{1} + \cdots + \Pi_{n_{0}}$. Let
\begin{equation*}
g_{k} (x) : = \frac{\chi_{k} ( x ) e^{- ( \phi (x) - \phi ( \m_{k} ) ) / h}}{\Vert \chi_{k} e^{- ( \phi - \phi ( \m_{k} ) ) / h} \Vert} .
\end{equation*}
From \eqref{z8}, we immediately get $g_{k} = f_{k}^{W , (0)} + \ooo (h)$. Moreover, as in \eqref{z9}, we have
\begin{equation*}
\big\Vert f_{k}^{(0)} - f_{k}^{W , (0)} \big\Vert = \big\Vert ( a_h^{-1} - 1 ) f_k^{W,(0)} \big\Vert = \ooo (h) .
\end{equation*}
Combining with Lemma \ref{d5}, we deduce
\begin{equation}
g_{k} = e_{k}^{(0)} + \ooo (h) .
\end{equation}
Using one more time Lemma \ref{d5}, the bases $( e_k^{(0)} )_{k}$ and $( g_{k} )_{k}$ of $\Im \Pi^{(0)}$ and $\Im \Pi$ respectively are almost orthogonal in the sense that
\begin{equation*}
\big\< e_{k}^{(0)} , e_{k '}^{(0)} \big\> = \delta_{k , k '} + \ooo ( h ) \qquad \text{and} \qquad \< g_{k} , g_{k'} \> = \delta_{k , k '} + \ooo ( h ) .
\end{equation*}
It then yields
\begin{equation}
\Pi = \Pi^{(0)} + \ooo (h) ,
\end{equation}
and \eqref{z2} follows from \eqref{z7}.
\end{proof}

\appendix
\section{Quasimodes, truncation procedure and labelling} \label{d1}

In this appendix, we gather from \cite{HeKlNi04_01} and \cite{HeHiSj11_01} the refined construction of quasimodes on $0$-forms for the Witten Laplacian, and the labeling procedure linking each minima with a saddle point of index $1$. We recall briefly the construction proposed in \cite{HeHiSj11_01} (which was in the Fokker-Planck case there) but in a generic situation where all $\phi(\s) - \phi(\m)$ are distinct, for $\m$, and $\s$ are respectively in the set of minima and saddle points of $\phi$.

In the following, we will denote by $\lll(\sigma)=\{x\in\R^n ; \ \phi(x)<\sigma\}$ the sublevel set  associated to the value $\sigma\in\R$.
Let $\s$ be a saddle point of $\phi$ and  $B ( \s , r ) = \{ x \in \R^{n} ; \ \vert x - s \vert <r \}$. Then, for $r > 0$ small enough, the set
\begin{equation*}
B(\s,r)\cap\lll(\phi(\s))=\{ x \in B ( \s , r ) ; \ \phi ( x ) < \phi ( \s ) \} ,
\end{equation*}
has precisely $2$ connected components, $C_{j} ( \s , r )$, with $j = 1 , 2$.

\begin{definition}\sl
We say that $\s \in \R^n$ is a separating saddle point (ssp) if it is either $\infty$  or it is a usual saddle point such that $C_{1} ( \s , r )$ and $C_{2} ( \s , r )$ are contained in different connected components of the set $\{ x \in \R^{n} ; \ \phi ( x ) < \phi ( \s ) \}$. We denote by  $\mathrm{SSP}$ the set of ssps.

We also introduce the set of separating saddle values ssv, $\mathrm{SSV} = \{ \phi ( \s ) ; \ \s \in \mathrm{SSP} \}$ with the convention that $\phi(\infty)=+\infty$.

A connected component $E$ of the sublevel set $\lll(\sigma)$ will be called a critical component if either $\partial E \cap \mathrm{SSP} \neq \emptyset$ or if $E = \R^n$.
\end{definition}

Let us now explain the way we label the critical points. We first order the saddle points in the following way. We recall from \cite{HeKlNi04_01} that $\sharp \mathrm{SSV}=n_0$ and then enumerate the ssvs in a decreasing way: $\infty=\sigma_1>\sigma_2>\cdots>\sigma_{n_0}$. To each ssv $\sigma_j$ we can associate a unique ssp: we define $\s_1=\infty$ and for any $j=2,\ldots,n_0$ we let
$\s_j$ be the unique ssp such that $\phi(\s_j)=\sigma_j$ (note that this $\s_j$ is unique thanks to Hypothesis \ref{h2}).

\medskip

Then we can proceed to the labelling of minima. We denote by $\m_1$ the global minimum of $\phi$,  $E_1 = \R^d$ and  by $S_1 = \phi(\s_1) - \phi(\m_1) = +\infty$ the critical Aarhenius value .

 Next we observe that  the sublevel set $\lll (\sigma_2)  = \{ x \in \R^{n} ; \ \phi(x)<\sigma_2\}$ is the union of two critical components, with one containing $\m_1$.
 The remaining connected component of the sublevel set
$\lll (\sigma_2)$ will be denoted by $E_{2}$ and its minimum $\m_2$. To the pair $(\m_2,\s_2)$ of critical points we associate the Arrhenius value $S_2 =  \phi(\s_2) - \phi(\m_2)$.
\par
Continuing the labelling procedure, we decompose the sublevel set $\lll (\sigma_3)$ into its connected components and perform the labelling as follows: we omit all
those components that contain the already  labelled minima $\m_{1}$ and $\m_{2}$. Some of these components may be non-critical.
There is only one critical one  remaining, and we denote it by $E_{3}$.  We then let $\m_{3}$ be the  point of global minimum of the restriction of $\phi$ to $E_{3}$ and
$S_3 =  \phi(\s_3) - \phi(\m_3)$.

We go on with this procedure, proceeding in the order dictated by the elements of the set $\mathrm{SSV}$, arranged in the decreasing
order, until all $n_0$  local minima $\m$ have been enumerated. In this way we have associated each local minima to one ssp: to each local
minimum $\m_k$, there is one critical component $E_k$, containing $\m_k$, and one ssp $\s_k$. We emphasize that in this procedure some  of the saddle points (the non critical ones)  may  not have been enumerated.
For convenience, we enumerate these remaining saddle points from $n_0+1$ to $n_1+1$. Note that
with this labelling $\uuu^{(1)} = \{ \s_2 , \ldots , \s_{n_1+1} \}$.  We have then
\begin{equation*}
\text{minima} = \{ \m_1 , \ldots , \m_{n_0} \} , \qquad \mathrm{SSP} = \{ \s_1 = \infty, \s_2 , \ldots , \s_{n_0} \} .
\end{equation*}
We summarize the preceding discussion in the following proposition:

\begin{proposition}\sl \label{e6}
The families of minima  $\uuu^{(0)} = \{ \m_k ; \ k = 1 , \ldots , n_0 \}$,  separating saddle points $\{ \s_k ; \ k = 1 , \ldots , n_0 \}$ and connected sets $\{ E_k ; \ k = 1 , \ldots , n_0 \}$ satisfy the following
\begin{enumerate}[i)]
\item We have $\s_1 = \infty$, $E_1=\R^n$ and $\m_1$ is the global minimum of $\phi$.
\item For every $k\geq 2$, $\overline{E_k}$ is compact,
$E_k$ is the connected component containing $\m_k$ in
$$
\{ x \in \R^{n} ; \ \phi ( x ) < \phi ( \s_{k} ) \} ,
$$
and $\phi (\m_k)=\min_{E_k}\phi $.
\item If $\s_{k'} \in E_k$   for some $k , k' \in \{ 1 , \ldots , n_0 \}$, then $k'>k$.
\end{enumerate}
\end{proposition}
In oder that the eigenvalues $\lambda_k^\star$ are decreasing , we eventually relabel the pairs of minima and critical saddle points so  that the sequence $S_k$ is decreasing.

\medskip

Using \cite{HeKlNi04_01} and \cite{HeHiSj11_01}, we shall now introduce suitable refined quasimodes,
adapted to the local minima of $\phi$ and the simplified labelling, described in Proposition \ref{e6}.
Let $\varepsilon_0 >0$ be such that the distance between critical points is larger than $10 \varepsilon_0$, and such
that for every critical point ${\bf u}$ and $k \in \{ 1 , \ldots , n_0 \}$ we
have either $\u \in \overline{E_k}$ or $\dist ( \u ,\overline{E_k})\geq 10 \varepsilon_0$. Let also  $C_0 >1$ to be defined
later,   and note that $\varepsilon_0$ may also be taken smaller later.
For $0 < \varepsilon < \varepsilon_0$ we build a family of functions $\chi_{k,\varepsilon}$, $k \in \{ 1 , \ldots , n_0 \}$ as follows: for $k=1$,
 we pose $\chi_{1,\varepsilon}=1$ and  for $k \geq 2$, we consider the open
set $E_{k,\varepsilon} = E_k \setminus \overline{B(\s_k, \varepsilon)}$, and let
$\chi_{k,\varepsilon}$ be a $C_0^\infty$-cutoff function supported in $E_{k,\varepsilon} + B(0,\varepsilon/C_0)$ and
equal to $1$ in $E_{k,\varepsilon} + B(0,\varepsilon /(2C_0))$.
Then, we define the quasimodes
for $1 \leq k \leq n_0$ by
\begin{equation} \label{z8}
f_k^{W,(0)}=  b_k(h) \chi _{k,\varepsilon} (x) e^{-(\phi (x)-\phi (\m_k))/h},
\end{equation}
where $b_k$ is a normalization constant, given thanks to the stationary phase theorem by
$$
b_k(h) = (\pi h)^{-d/4} \det(\hess \phi(\m_k))^{1/4} (b_{k,0} + h b_{k,1} + \cdots ), \ \ \ \ b_{k,0} =1.
$$
Then for $\varepsilon_0$ small
enough and $C_0$ large enough, there exists $C>0$ such that for all
$0< \varepsilon < \varepsilon_0$, we have the following lemma:
\begin{lemma}\sl \label{e7}
 The system $(f_k^{W,(0)})$ is free and there exist $\alpha >0$ uniform in $\varepsilon <\varepsilon_0$,  such that
\begin{equation*}
\big\< f_k^{W,(0)},f_{k'}^{W,(0)} \big\> = \delta_{k,k'} + \ooo(e^{-\alpha/h}), \qquad d_{\phi,h} f_k^{W,(0)}  =  \ooo(e^{-(S_k-C\varepsilon)/h}) ,
\end{equation*}
and in particular
\begin{equation*}
P^{W,(0)} f_k^{W,(0)} = \ooo(e^{-\alpha/h}).
\end{equation*}
\end{lemma}

\begin{proof} This Lemma is a direct consequence of  the statement and the proof of Proposition 5.3 in  \cite{HeHiSj11_01}.
\end{proof}

\bibliographystyle{amsplain}
\providecommand{\bysame}{\leavevmode\hbox to3em{\hrulefill}\thinspace}
\providecommand{\MR}{\relax\ifhmode\unskip\space\fi MR }
\providecommand{\MRhref}[2]{%
  \href{http://www.ams.org/mathscinet-getitem?mr=#1}{#2}
}
\providecommand{\href}[2]{#2}

\end{document}